\documentclass[12pt,a4paper]{amsart}
\usepackage{amsfonts}
\usepackage{amsthm}
\usepackage{amsmath}
\usepackage{amscd}
\usepackage[latin2]{inputenc}
\usepackage{t1enc}
\usepackage[mathscr]{eucal}
\usepackage{indentfirst}
\usepackage{graphicx}
\usepackage{graphics}
\numberwithin{equation}{section}
\usepackage[margin=2.9cm]{geometry}
\usepackage{epstopdf} 
\usepackage{xcolor}
 
\usepackage{comment}

\theoremstyle{plain}
\newtheorem{theorem}{Theorem}[section]
\newtheorem{lemma}[theorem]{Lemma}

\newtheorem{proposition}[theorem]{Proposition}

 \theoremstyle{definition}

\newtheorem{remark}[theorem]{Remark}
\newtheorem{?}[theorem]{Problem}

\newcommand{\R}{\mathbb{R}}
\newcommand{\II}{{\rm{II}}}
\newcommand{\D}{\mathbb{D}}

\newcommand{\e}{{\rm {exp}}}
\newcommand{\A}{\mathcal A}

\begin{document}

\title[Narrow escape problem]{Narrow escape problem in the presence of the force field}


\author[Medet Nursultanov]{Medet Nursultanov}
\address {School of Mathematics and Statistics, University of Sydney}
\email{medet.nursultanov@gmail.com}

\author[William Trad] {William Trad}
\address {School of Mathematics and Statistics, University of Sydney}
\email{w.trad@maths.usyd.edu.au}

\author[Leo Tzou]{Leo Tzou}
\address {School of Mathematics and Statistics, University of Sydney}
\email{leo.tzou@gmail.com}

\thanks{M. Nursultanov, W. Trad, and L. Tzou are partially supported by ARC DP190103302 and ARC DP190103451 during this work. This work forms part of the second author's PhD thesis under the supervision of the third author.}



\maketitle
\begin{abstract}
	This paper considers the narrow escape problem of a Brownian particle within a three-dimensional Riemannian manifold under the influence of the force field. We compute an asymptotic expansion of mean sojourn time for Brownian particles. As a auxiliary result, we obtain the singular structure for the restricted Neumann Green's function which may be of independent interest. 
\end{abstract}

\section{Introduction}

Let us consider a Brownian particle confined to a bounded domain by a reflecting boundary, except for a small absorbing part which is thought of as a target. The narrow escape problem deals with computing the mean sojourn time of the aforementioned Brownian particle. Mathematically, this can be formulated as follows. Let $(M,g,\partial M)$ be a compact, connected, orientable Riemannian manifold with non-empty smooth boundary $\partial M$. Additionally, let $(X_t, \mathbb{P}_x)$ be the Brownian motion on $M$ generated by differential operator 
$$\Delta_{g} \cdot +g(F,\nabla_g\cdot)$$
where $\Delta_{g}=-d^*d$ is the (negative) Laplace-Beltrami operator,  $\nabla_g$ is the gradient, $F$ is a force field given by the potential $\phi$, that is $F = \nabla_g\phi$. We use $\Gamma_\varepsilon\subset \partial M$ to denote the absorbing window through which the $(X_t,\mathbb{P}_t)$ can escape, we use $\varepsilon$ to denote the size of the window and we denote by $\tau_{\Gamma_{\varepsilon}}$ the first time the Brownian motion $X_t$ hits $\Gamma_{\varepsilon}$, that is
\begin{equation*}
\tau_{\Gamma_{\varepsilon}} := \inf \{ t\geq 0: X_t \in \Gamma_{\varepsilon}\}.
\end{equation*}
As said earlier, we wish to derive asymptotics as $\varepsilon \to 0$ for the mean sojourn time which is denoted by $u_\varepsilon$ and is given by $\mathbb{E}[\tau_{\Gamma_{\varepsilon}} | X_0 = x]$. Another quantity of interest is the {\em spatial average} of the mean sojourn time:
$$|M|^{-1} \int_M \mathbb{E}[\tau_{\Gamma_{\varepsilon, a}} | X_0 = x] d_g(x).$$
Here $|M|$ denotes the Riemannian volume of $M$ with respect to the metric $g$.

Initially, this problem was mentioned in the context of acoustics in \cite{Rayleigh} (1945). Much later (2004), the interest to this problem was renewed due to relation to molecular biology and biophysics; see \cite{holcman2004escape}. Many problems in cellular biology may be formulated as mean sojourn time problems; a collection of analysis methods, results, applications, and references may be found in \cite{holcman2015stochastic} and \cite{BressloffNewby}. For example, cells have been modelled as simply connected two-dimensional domains with small absorbing windows on the boundary representing ion channels or target binding sites; the quantity sought is then the mean time for a diffusing ion or receptor to exit through an ion channel or reach a binding site \cite{schuss2007narrow, holcman2004escape, pillay2010asymptotic}. All of this lead to the narrow escape theory in applied mathematics and computational biology; see \cite{schuss2006ellipse,singer2006narrow,SingerSchussHolcman2006}.

There has been much progress for this problem in the setting of planar domains, and we refer the readers to \cite{holcman2004escape, pillay2010asymptotic, singer2006narrow, ammari} and references therein for a complete bibliography. An important contribution was made in the planar case by \cite{ammari} to introduce rigor into the computation of \cite{pillay2010asymptotic}. The use of layered potential in \cite{ammari} also cast this problem in the mainstream language of elliptic PDE and facilitates some of the approach we use in this article. 

Few results exists for three dimensional domains in $\R^n$ or Riemannian manifolds; see \cite{cheviakov2010asymptotic,schuss2006ellipse,singer2008narrow, GomezCheviakov, NTT} and references therein. The additional difficulties introduced by higher dimension are highlighted in the introduction of \cite{ammari} and the challenges in geometry are outlined in \cite{singer2008narrow}. In the case when $M$ is a domain in $\R^3$ with Euclidean metric and $\Gamma_{\varepsilon, a}$ is a single small disk absorbing window, \cite{schuss2006ellipse, singer2008narrow} gave an expansion for the average of the expected first arrival time, averaged over $M$, up to an unspecified $O(1)$ term.
These results were improved upon in \cite{NTT}, by using geometric microlocal analysis. Namely, the authors derived the bounded term and estimated the remaining term, moreover, they obtain these results for general Reimaniann manifolds. The case when $\Gamma_{\varepsilon,a}$ is a small elliptic window was also addressed in \cite{schuss2006ellipse, singer2008narrow,NTT}. We also mention \cite{Schuss2012}, where the author gave a short review of related works (up to 2012).

When $M$ is a three dimensional ball with multiple circular absorbing windows on the boundary, an expansion
capturing the explicit form of the $O(1)$ correction in terms of the Neumann Green's function and its {\em regular part} was done 
in \cite{cheviakov2010asymptotic}. The method of matched asymptotic used there required the explicit computation of the Neumann Green's function, which is only possible in special geometries with high degrees of symmetry/homogeneity. In these results one does not see the full effects of local geometry. This result was also rigorously proved in \cite{ChenFriedman} but with a better estimate for the error term.

Much less has been done for the case of non-zero force field. For instance, all works we metioned above, except \cite{ammari}, deal with the diffusion without a force field, that is $F=0$. We could find two works concerning this case: \cite{SingerSchuss2007} and \cite{ammari}. Both these works consider $M$ being a domain in $\mathbb{R}^2$ or $\mathbb{R}^3$. In \cite{SingerSchuss2007}, the authors generalize the method of \cite{schuss2006ellipse,singer2006narrow,SingerSchussHolcman2006} to obtain the leading-order term of the average of the expected first arrival time for two and three dimensional cases. For planar domains, the authors in \cite{ammari}, by using layer potential techniques, derive assymptotic expansion up to $O(\varepsilon)$ term.

In this paper, we derive all the main terms of the expected value of the first arrival time for Riemaniann manifolds of dimension three in the presence of a force field. The window or target, $\Gamma_{\varepsilon}$, is considered to be a small geodesic ellipse of eccentricity $\sqrt{1-a^2}$ and size $\varepsilon \to 0^+$ (to be made precise later). To investigate $\mathbb{E}[\tau_{\Gamma_{\varepsilon}} | X_0 = x]$, we needed to know a singularity structure of the Neumann Green's function, which is given by the equation
\begin{equation}\label{NGF System}
\begin{cases}
\Delta_{g,z} G(x,z) - \mathrm{div}_{g,z} (F(z)G(x,z)) = -\delta_x(z);\\
\partial_{\nu_z}G(x,z) - F(z)\cdot \nu_z G(x,z) \left.\right|_{\partial M} = -\frac{1}{|\partial M|};\\
\int_{\partial M}e^{-\phi(z)} G(x,z) d_h(z) = 0.
\end{cases}
\end{equation}
is required. For the case $F=0$, the authors in \cite{singer2008narrow} highlighted the difficulty in obtaining a comprehensive singularity expansion of $G(x,z)\mid_{x,z\in\partial M, x\neq z}$ in a neighbourhood of the diagonal $\{x=z\}$ when $M$ is a bounded domain in $\R^n$, but it turns out that even when $M$ is a general Riemannian manifold this question can be treated via pseudo-differential techniques, which is done in \cite{NTT}. Here, we generalize result of \cite{NTT} for the case $F\neq 0$. Knowledge of the singularity structure allows us to derive the mean first arrival time of a Brownian particle on a Riemannian manifold with a single absorbing window which is a small geodesic ellipse. Our method extends to multiple windows but we present the single window case to simplify notations.

The paper is organized as follows. In Section \ref{notations}, we introduce the notations. In Section \ref{results}, we formulate the problem, state and discuss the main results of this paper. Section \ref{Green} deals with computing the singular structure of the Neumann Green's function. Finally, in Section \ref{pr main rres} we carry out the asymptotic calculation using the tools we have developed. The appendix characterizes the expected first arrival time $\mathbb{E}[\tau_{\Gamma_{\varepsilon, a}} | X_0 = x]$ as the solution of an elliptic mixed boundary value problem. This is classical in the Euclidean case (see \cite{schussbook}) but we could not find a reference for the general case of a Riemannian manifold with boundary.

	\section{Preliminaries}\label{notations}
	\subsection{Notation}
	Throughout  this  paper, $(M,g,\partial M)$ be a compact connected orientable Riemannian manifold with non-empty smooth boundary. The corresponding volume and geodesic distance are denoted by $d_g(\cdot)$ and $d_g(\cdot,\cdot)$, respectively. By $|M|$ we denote the volume of $M$.
	
	Let $\iota_{\partial M}: \partial M \hookrightarrow M$ is the trivial embedding of the boundary $\partial M$ into $M$. This allows us to define the following metric $h := \iota_{\partial M}^*g$ be the metric on the boundary $\partial M$. Set $d_h(\cdot)$ and $d_h(\cdot,\cdot)$ to be the volume and the geodesic distance on the boundary given by metric $h$. By $|\partial M|$ we denote the volume of $\partial M$ with respect to $d_h$.
	
	For $x\in \partial M$, let $E_1(x), E_2(x) \in T_{x}\partial M$ be the unit eigenvectors of the shape operator at $x$ corresponding respectively to the principal curvatures $\lambda_1(x),\ \lambda_2(x)$. We will drop the dependence in $x$ from our notation when there is no ambiguity. We choose $E_1$ and $E_2$ such that $E_1^\flat\wedge E_2^\flat\wedge\nu^\flat$ is a positive multiple of the volume form $ d_g$ (see p.26 of \cite{lee} for the ``musical isomorphism'' notation of $^\flat$ and $^\sharp$). Here we use $\nu$ to denote the outward pointing normal vector field. By $H(x)$, we denote the mean curvature of $\partial M$ at $x$. We also set 
	$$\II_x(V) := \II_x(V,V),\ \ V\in T_x\partial M$$
	to be the scalar second fundamental quadratic form (see pages 235 and 381 of \cite{lee} for definitions). Note that, in defining $\II$ and the shape operator, we will follow geometry literature (e.g.\cite{lee}) and use the {\em inward} pointing normal so that the sphere embedded in $\R^3$ would have positive mean curvature in our convention.
	
	\subsection{Boundary normal coordinates}
	In this work, we will often use the boundary normal coordinates. Therefore, we briefly recall its construction. For a fixed $x_0\in \partial M$, we will denote by $B_h(\rho;x_0) \subset \partial M$ the geodesic disk of radius $\rho>0$ (with respect to the metric $h$) centered at $x_0$ and $\D_\rho$ to be the Euclidean disk in $\R^2$ of radius $\rho$ centered at the origin. In what follows $\rho$ will always be smaller than the injectivity radius of $(\partial M, h)$. Letting $t = (t_1,t_2,t_3) \in \R^3$, we will construct a coordinate system $x(t; x_0)$ by the following procedure:
	
	Write $t\in \R^3$ near the origin as $t  = (t',t_3)$ for $t' = (t_1,t_2)\in \D_\rho$. Define first 
	$$x((t',0); x_0) := {\rm {exp}}_{x_0;h} (t_1 E_1+ t_2 E_2),$$
	where ${\rm{exp}}_{x_0;h} (V)$ denotes the time $1$ map of $h$-geodesics with initial point $x_0$ and initial velocity $V\in T_{x_0}\partial M$. The coordinate $t'\in \D_\rho \mapsto x((t',0); x_0)$ is then an $h$-geodesic coordinate system for a neighborhood of $x_0$ on the boundary surface $\partial M$. We can extend this to become a coordinate system for points in $M$ near $x_0$ so that $t \mapsto x(t;x_0)$ is a boundary normal coordinate system with $t_3 >0$ in $M$ as the boundary defining function. Readers wishing to know more about boundary normal coordinates can refer to \cite{leeuhlmann} for a brief recollection of the basic properties we use here and Prop 5.26 of \cite{lee} for a detailed construction.  
	
	For convenience we will write $x(t'; x_0)$ in place of $x((t',0);x_0)$. The boundary coordinate system $ t\mapsto x(t;x_0)$ has the advantage that the metric tensor $g$ can be expressed as
	\begin{eqnarray}
	\label{metric bnc}
	\sum_{j,k=1}^3 g_{j,k}(t) dt_j dt_k = \sum_{\alpha, \beta = 1}^2 h_{\alpha,\beta}(t',t_3) dt_\alpha dt_\beta + dt_3^2,
	\end{eqnarray}
	where $h_{\alpha,\beta}(t', 0) = h_{\alpha,\beta}(t')$ is the expression of the boundary metric $h$ in the $h$-geodesic coordinate system $x(t';x_0)$. 
	
	We will also use the rescaled version of this coordinate system. For $\varepsilon >0$ sufficiently small we define the (rescaled) $h$-geodesic coordinate by the following map
	\begin{equation}\label{res coord}
		x^{\varepsilon}(\cdot ; x_0) : t' = (t_1, t_2) \in \D \mapsto x(\varepsilon t'; x_0) \in B_h(\varepsilon;x_0),
	\end{equation}
	where $ \mathbb{D}$ is the unit disk in $\R^2$.

	\section{The main results}\label{results}
	Here we state and disscus the main results of this paper. We begin with formulating the problem. Let $(X_t, \mathbb{P}_x)$ be the Brownian motion on $M$ starting at $x$, generated by the differential operator
	\begin{equation*}
	u \rightarrow \Delta_gu + g(F,\nabla_gu),
	\end{equation*}
	where $F$ is a force field given by potential $\phi$, that is $F=\nabla_g\phi$. For $x^*\in \partial M$ and  $\varepsilon>0$, let $\Gamma_{\varepsilon,a}\subset \partial M$ be a small geodesic ellipse define as
	\begin{equation}\label{ellipse}
		\Gamma_{\varepsilon, a} := \{\exp_{x^*;h}(\varepsilon t_1 E_1(x^*) +\varepsilon t_2 E_2(x^*)) \mid t_1^2 + a^{-2} t_2^2 \leq 1 \}.
	\end{equation}
	Denote by $\tau_{\Gamma_{\varepsilon,a}}$ the first time the Brownian motion $X_t$ hits $\Gamma_{\varepsilon,a}$, that is
	\begin{equation*}
	\tau_{\Gamma_{\varepsilon,a}} := \inf \{ t\geq 0: X_t \in \Gamma_{\varepsilon,a}\}.
	\end{equation*}
	We aim to investigate the mean sojourn time, that is the expected value
	$$u_{\varepsilon,a}(x) := \mathbb{E}[\tau_{\Gamma_{\varepsilon, a}} | X_0 = x],$$
	and its avarege expected value over $M$
	$$|M|^{-1} \int_M \mathbb{E}[\tau_{\Gamma_{\varepsilon, a}} | X_0 = x] d_g(x).$$
	Namely, we want to derive asymptotic expansion for these quantities as $\varepsilon\to 0$.
	
	In the Appendix, we show that the mean sojourn time, $u_{\varepsilon,a}$, satisfies the following elliptic mix boundary value problem
	\begin{equation}\label{main bvp}
		\begin{cases}
		\Delta_gu_{\varepsilon,a} + g(F,\nabla_gu_{\varepsilon,a}) = -1; \\
		\left.u_{\varepsilon,a}\right\vert_{\Gamma_{\varepsilon,a}}=0;\\
		\left.\partial_\nu u_{\varepsilon,a}\right\vert_{\partial M\setminus \Gamma_{\varepsilon,a}}=0.
		\end{cases}
	\end{equation}
	
	Let $G(\cdot,\cdot)\in \mathcal{D}'(M\times M)$ solve \eqref{NGF System}. For $x\in M^o$, Greens formula used in conjunction with \eqref{NGF System} and \eqref{main bvp} yields the following integral representation for the mean sojourn time $u_{\varepsilon,a}$	
	\begin{align}\label{org_int_eq_main}
	u_{\varepsilon,a}(x) = \mathcal{G}(x) + \int_{\partial M}G(x,z) \partial_{\nu_z} u_{\varepsilon,a}(z) d_h(z) + C_{\varepsilon,a}
	\end{align}
	where 
	\begin{equation*}
	C_{\varepsilon,a} := \frac{1}{|\partial M|} \int_{\partial M} u_{\varepsilon,a}(z) d_h(z), \qquad \mathcal{G}(x) := \int_M G(x,z) d_h(z).
	\end{equation*}
    Where $\mathcal{G}$ satisfies the following boundary value problem
	\begin{equation}\label{eq for mathcal G for theorem}
		\begin{cases}
		\Delta_{g} \mathcal{G}(x) + g ( F(x), \nabla_{g}\mathcal{G}(x) ) = -1;\\
		\partial_{\nu}\mathcal{G}(x)\left.\right|_{x\in\partial M} = -\frac{\Phi(x)}{|\partial M|};\\
		\int_{\partial M} \mathcal{G}(x) d_h(x) =0.
		\end{cases}
	\end{equation}
	Where $\Phi$ is a weighted volume defined in Theorem 3.2. 
	
	In order to derive an asymptotic expansion for $u_{\varepsilon,a}$, we need to derive asymptotics for $C_{\varepsilon,a}$. The first step within said program is to exploit the vanishing Dirichlet boundary condition of $\Gamma_{\varepsilon,a}$. In doing so, we restrict $x\in M^o$ to $x\in \Gamma_{\varepsilon,a}$ which yields
	\begin{align*}
	    0 = \left.\mathcal{G}(x)\right\vert_{\Gamma_{\varepsilon,a}}+\left.\left(\int_{\partial M} G(x,z) \partial_{\nu_z}u_{\varepsilon,a}(z)d_h(z)\right)\right\vert_{\Gamma_{\varepsilon,a}} +C_{\varepsilon,a}.
	\end{align*}
	What follows, is the definition of the restricted Neumann Greens function, defined as the Schwartz kernel to the operator 
	\begin{align*}
	    G_{\partial M}: f \mapsto \left.\left(\int_{\partial M} G(x,y) f(y) d_h(y)\right)\right\vert_{\partial M}.
	\end{align*}
	Here $G_{\partial M}: C^\infty(\partial M)\rightarrow C^\infty(\partial M)$ can be extended to $H^k(\partial M)$. Using a parametrix construction, in conjunction with Fourier techniques and homogeneous distributions, we can show that the kernel $G_{\partial M}$ attains the following form for $x,y\in \partial M$ near the diagonal.  
	\begin{proposition}\label{G exp sec 3}
		There exists an open neighbourhood of the diagonal 
		$${\rm Diag}:=\{(x,y)\in \partial M\times \partial M \mid x=y\} $$
		such that in this neighbourhood, the singularity structure of $G_{\partial M}(x,y)$ is given by:
		\begin{align}\label{G expansion}
			\nonumber G_{\partial M}(x,y) = &\frac{1}{2\pi} d_g(x,y)^{-1} - \frac{1}{4\pi}  {(H(x) + \partial_{\nu}\phi(x))} \log d_{h}(y,x) \\
			& + \frac{1}{16\pi} \left(\II_x\left (\frac{\exp_{x; h}^{-1} (y)}{|\exp_{x;h}^{-1} (y)|_h}\right) - \II_x\left (\frac{*\exp_{x;h}^{-1} (y)}{|\exp_{x;h}^{-1} (y)|_h}\right)\right)\\
			\nonumber& -\frac{1}{4\pi}h_x\left(F^{\parallel}(x), \frac{\exp_{x; h}(y)}{|\exp_{x; h}(y)|_h}\right)  + R(x,y),
		\end{align}
		where $F^{\parallel}$ is the tangential part of the force field $F$ and $R(\cdot,\cdot) \in C^{0,\mu}(\partial M\times \partial M)$, for all $\mu<1$, is called {\em the regular part} of the Green's function and $*$ is the Hodge-star operator (i.e. rotation by $\pi/2$ on the surface $\partial M$).
	\end{proposition}
	We will use the formula in Proposition \ref{G exp sec 3} to derive the mean first arrival time of a Brownian particle on a Riemannian manifold with a single absorbing window which is a small geodesic ellipse. As mentioned earlier, our method extends to multiple windows but we present the single window case to simplify notations. We first state the result when the window is a geodesic {\em disk} of the boundary $\partial M$ around a fixed point since the statement is cleaner:
	\begin{theorem}
		\label{main theorem disk}
		Let $(M,g,\partial M)$ be a smooth compact Riemannian manifold of dimension three with boundary. Fix $x^*\in \partial M$ and let $\Gamma_{\varepsilon}$ be a boundary geodesic ball centered at $x^*$ of geodesic radius $\varepsilon >0$.\\
		i)  For each $x\notin \Gamma_{\varepsilon}$,
		$$ \mathbb{E}[\tau_{\Gamma_{\varepsilon}} | X_0 = x] = C_\varepsilon + \mathcal{G}(x) -\Phi(x^*) G(x^*,x) + r_\varepsilon (x),$$
		with $\|r_\varepsilon\|_{C^k(K)} \leq C_{k,K} \varepsilon$ for any integer $k$ and compact set $K\subset \overline M$ which does not contain $x^*$. The function $\mathcal{G}$ is the solution of \eqref{eq for mathcal G for theorem}. The constant $C_\varepsilon$ is given by
		\begin{align*}
		C_{\varepsilon} = \frac{\Phi(x^*)}{4\varepsilon} - \frac{(H(x^*) + \partial_{\nu}\phi(x^*)) \Phi(x^*)}{4\pi} \log\varepsilon + R(x^*, x^*) \Phi(x^*) - \mathcal{G}(x^*)\\
		- \frac{(H(x^*) + \partial_{\nu}\phi(x^*)) \Phi(x^*)}{4\pi} \left( 2\log 2 - \frac{3}{2}\right) + O(\varepsilon\log\varepsilon),
		\end{align*}
		where $R(x^*, x^*)$ is the evaluation at $(x,y) = (x^*, x^*)$ of the kernel $R(x,y)$ in Proposition \ref{G exp sec 3} and
		$$\Phi(x) := \int_{M} e^{\phi(z) - \phi(x)}d_g(z).$$
		ii) One has that the integral of $\mathbb{E}[\tau_{\Gamma_{\varepsilon, a}} | X_0 = x]$ over $M$ satisfies 
		$$\int_M \mathbb{E}[\tau_{\Gamma_{\varepsilon, a}} | X_0 = x] d_g (x)= C_\varepsilon |M| + \int_M \mathcal{G}(x) d_g(x) - \Phi(x^*) \int_{M} G(x, x^*) d_g(x).$$
	\end{theorem}

Theorem \ref{main theorem disk} does not realize the full power of Proposition \ref{G exp sec 3} as it does not see the non-homogeneity of the local geometry at $x^*$ (only the mean curvature $H(x^*)$ shows up). This is due to the fact that we are looking at windows which are geodesic balls. If we replace geodesic balls with geodesic {\em ellipses}, we see that the second fundamental form term in \eqref{G expansion} contributes to a term in $\mathbb{E}[\tau_{\Gamma_{\varepsilon, a}} | X_0 = x]$ which is the {\em difference} of principal curvatures. 

\begin{theorem}
	\label{main theorem}
	Let $(M,g,\partial M)$ be a smooth Riemannian manifold of dimension three with boundary. Fix $x^*\in \partial M$ and let $\Gamma_{\varepsilon,a}$ be a boundary geodesic ellipse given by \eqref{ellipse} with $\varepsilon >0$.\\
	i) For each $x\in M\backslash \Gamma_{\varepsilon,a}$, 
	$$ \mathbb{E}[\tau_{\Gamma_{\varepsilon}} | X_0 = x] = C_\varepsilon + \mathcal{G}(x) -\Phi(x^*)G(x^*,x) + r_\varepsilon (x),$$
	with $\|r_\varepsilon\|_{C^k(K)} \leq C_{k,K} \varepsilon$ for any integer $k$ and compact set $K\subset \overline M$ which does not contain $x^*$. The function $\mathcal{G}$ is the solution of \eqref{eq for mathcal G for theorem}. The constant $C_\varepsilon$ is given by
	\begin{align*}
	C_{\varepsilon,a} = & \frac{K_a\Phi(x^*)}{4\pi^2 a \varepsilon} - \frac{(H(x^*) + \partial_{\nu}\phi(x^*)) \Phi(x^*)}{4\pi} \log\varepsilon\\
	& + R(x^*, x^*) \Phi(x^*) - \mathcal{G}(x^*)\\
	& - \frac{(H(x^*) + \partial_{\nu}\phi(x^*))\Phi(x^*)}{16 \pi^3}\int_{\mathbb{D}} \frac{1}{ (1-|s'|^2)^{1/2}} \int_\D  \frac{\log\left((t_1 - s_1)^2 + a^2 (t_2-s_2)^2 \right)^{1/2}}{ (1-|t'|^2)^{1/2}} dt' ds'\\
	&+\frac{(\lambda_1(x^*) - \lambda_2(x^*))\Phi(x^*)}{64\pi^3}\int_{\mathbb{D}} \frac{1}{ (1-|s'|^2)^{1/2}} \int_\D \frac{(t_1 - s_1)^2 - a^2 (t_2 - s_2)^2}{(t_1 - s_1)^2 + a^2 (t_2 - s_2)^2} \frac{1}{ (1-|t'|^2)^{1/2}} dt' ds'\\
	& + O(\varepsilon\log\varepsilon),
	\end{align*}
	where $R(x^*, x^*)$ is the evaluation at $(x,y) = (x^*, x^*)$ of the kernel $R(x,y)$ in Proposition \ref{G exp sec 3}, $\mathbb{D}$ is the two dimensional unit disk centered at the origin, and 
	\begin{equation*}
		\Phi(x) := \int_{M} e^{\phi(z) - \phi(x)}d_g(z), \qquad K_a: = \frac{\pi}{2} \int_{0}^{2 \pi} {\left( \cos^2 \theta + \frac{\sin^2 \theta}{a^2} \right)^{-1/2}} d \theta
	\end{equation*} 
	
	ii) One has that the integral of $\mathbb{E}[\tau_{\Gamma_{\varepsilon, a}} | X_0 = x]$ over $M$ satisfies 
		$$\int_M \mathbb{E}[\tau_{\Gamma_{\varepsilon, a}} | X_0 = x] d_g (x)= C_\varepsilon |M| + \int_M \mathcal{G}(x) d_g(x) - \Phi(x^*) \int_{M} G(x, x^*) d_g(x).$$
	
	
	
\end{theorem}

\section{The Neumann Green's Function}\label{Green}
Here, we investigate the Neumann Green's function. Namely, we derive its singular structure on the boundary near the diagonal. By Neumann Green's function, $G(\cdot,\cdot)\in \mathcal{D}'(M\times M)$, we mean the solution to the following equation
\begin{equation}\label{Green function equation}
\begin{cases}
\Delta_{g,z} G(x,z) - \mathrm{div}_{g,z} (F(z)G(x,z)) = -\delta_x(z);\\
\partial_{\nu_z}G(x,z) - F(z)\cdot \nu_z G(x,z) \left.\right|_{\partial M} = -\frac{1}{|\partial M|};\\
\int_{\partial M}e^{-\phi(z)} G(x,z) d_h(z) = 0.
\end{cases}
\end{equation}
By using Green's identity to
\begin{equation*}
	e^{-\phi(y)}G(z,y) = -\int_{M} e^{-\phi(x)}G(z,x) \left( \Delta_g G(y,x) - \text{div}_g(F(x)G(y,x))\right)d_g(x),
\end{equation*}
we obtain
\begin{equation*}
	e^{-\phi(y)}G(z,y) - e^{-\phi(z)}G(y,z) = \frac{1}{|\partial M|} \left(\int_{\partial M}e^{-\phi(x)} G(y,x) d_h(x) - \int_{\partial M}e^{-\phi(x)} G(z,x) d_h(x)\right).
\end{equation*}
Therefore, by the last condition in \eqref{Green function equation}, we obtain
\begin{equation}\label{sym for G}
	G(z,y) = e^{\phi(y)-\phi(z)}G(y,z).
\end{equation}
Therefore, we can check
\begin{align*}
	\int_{\partial M} G(z,y)d_h(z) = \int_{\partial M} e^{\phi(y)-\phi(z)} G(y,z) d_h(z) = 0,
\end{align*}
\begin{align*}
	\partial_{\nu_z} G(z,y) = e^{\phi(y)-\phi(z)}\big( \partial_{\nu_z}G(y,z) - F(z)\cdot \nu_z G(y,z) \left.\right|_{\partial M}  \big) = -\frac{e^{\phi(y)-\phi(z)}}{|\partial M|}
\end{align*}
for $z\in \partial M$, and finally
\begin{align*}
	\Delta_{g,z} G(z,y) + &g ( F(z), \nabla_{g,z}G(z,y) ) = e^{\phi(y)-\phi(x)} \Big(  g(\nabla_{g,z}\phi(z),\nabla_{g,z}\phi(z) )G(y,z)\\
	& - \Delta_{g,z} \phi(z) G(y,z) - 2g(\nabla_{g,z}\phi(z),\nabla_{g,z}G(y,z)) +  \Delta_{g,z} G(y,z)\\
	&  - g(F(z),\nabla_{g,z}\phi(z) )G(y,z) + g(F(z),\nabla_{g,z}G(y,z) )\Big).
\end{align*}
Therefore, we conclude that $G(z,y)$ satisfies the differential equation (with respect to the first variable)
\begin{equation}\label{adj}
\begin{cases}
	\Delta_{g,z} G(z,y) + g ( F(z), \nabla_{g,z}G(z,y) ) = - \delta_y(z);\\
	\partial_{\nu_z}G(z,y)\left.\right|_{z\in\partial M} = -\frac{e^{\phi(y)-\phi(z)}}{|\partial M|};\\
	\int_{\partial M} G(z,x) d_h(z) =0.
\end{cases}
\end{equation}
This indicates that $\mathcal{G}$ solves the problem \eqref{eq for mathcal G for theorem}. Now, we let $f \in C^\infty(\partial M)$. Using this $f$, we introduce $u_f$, the solution to the following auxiliary problem 
\begin{align}\label{BVP for DN}
\Delta_g u_f(z) + g_z(F(z),\nabla_g u_f(z)) = 0, \;\ \left.u_f\right\vert_{z\in\partial M} = f \in C^\infty(\partial M).
\end{align}
By using Green's identity and the Divergence form theorem to
\begin{align}
	u_f(x) &= -\int_{M} u_f(z) \left(\Delta_{g,z} G(x,z) - \mathrm{div}_{g,z} (F(z)G(x,z))\right)d_g(z),	
\end{align}
we compute
\begin{align*}
	u_f(x) = &-\int_{M} \Delta_{g}u_f(z) G(x,z) d_g(z) - \int_{\partial M} \left(u_f(z) \partial_{\nu_z} G(x,z) - \partial_{\nu_z} u_f(z)  G(x,z) \right) d_h(z)\\
	&-\int_{M} G(x,z) g_z(F(z),\nabla_{g}u_f(z))d_g(z) + \int_{\partial M} u_f(z) G(x,z) F(z) \cdot \nu_z d_h(z).
\end{align*}
Since functions $u_f$, $G(x,z)$ satisfy \eqref{BVP for DN}, \eqref{Green function equation} respectively, we conclude that
\begin{align*}
u_f(x) = \int_{\partial M} G(x,z)\partial_{\nu_z} u_f(z) d_h(z) + \frac{1}{|\partial M|}\int_{\partial M} f(z) d_h(z).
\end{align*}

Restricting $x$ to $\partial M$, we have that
\begin{align}\label{compose with DN map}
\left.f(x)\right\vert_{\partial M} &= \left.\left(\int_{\partial M} G(x,z) \partial_{\nu_z}u_f(z) d_h(z)\right)\right\vert_{\partial M} + \frac{1}{|\partial M|}\int_{\partial M}f(z) d_h(z)
\end{align}
Let $\Lambda_F\in \Psi^1_{cl}(\partial M)$ be the Dirichlet-to-Neumann map associated to the boundary value problem \eqref{BVP for DN} and $G_{\partial M}(x,y)$ be the Schwartz kernel of the operator
\begin{equation*}
f \rightarrow \left.\left(\int_{\partial M} f(y)G(x,y) d_h(y)\right)\right\vert_{\partial M}
\end{equation*}
which takes $C^\infty(\partial M) \rightarrow C^\infty(\partial M)$. Then we can rewrite \eqref{compose with DN map} in the following way
\begin{align*}
f(x) = G_{\partial M} \Lambda_F f + Pf
\end{align*}
where $P$ is a smoothing operator, that is $P\in \Psi^{-\infty}(\partial M)$. In operator form this is
\begin{align}\label{left parametrix}
I = G_{\partial M} \Lambda_F + P.
\end{align}
Since $\Lambda_F$ is an elliptic pseudo-differential operator, we can construct $G_{\partial M}$ via a standard left parametrix construction. 

\subsection{Symbolic Expansion for the symbol of the Dirichlet-to-Neumann map}\label{SND}
	We compute here the first two terms of the asymptotic expansion for the symbol of the Dirichlet-to-Neumann map. We will use this to obtain the corresponding terms for the symbol of $G_{\partial M}$. We will follow \cite{LeeUhlmann89} and adapt some of their results for the drift case.

	In boundary normal coordinates, we decompose our differential operator in the following way
	\begin{align}\label{Decomposition for Laplacian}
	-\Delta_g-g_x(F,\nabla_g\cdot) = D_{x^3}^2+i\widetilde{E}(x)D_{x^3}+\widetilde{Q}(x,D_{x'}) 
	\end{align}
	where
	\begin{align*}
	\widetilde{E}(x)&:=-\frac{1}{2}\sum_{\alpha,\beta}h^{\alpha\beta}(x)\partial_{x^n}h_{\alpha\beta}(x)-F^n(x)\\
	\widetilde{Q}(x,D_{x'})&:=\sum_{\alpha,\beta} h^{\alpha\beta}(x)D_{x^\alpha}D_{x^\beta}\\ &- i\sum_{\alpha,\beta} \left(\frac{1}{2}h^{\alpha\beta}(x)\partial_{x^\alpha}\log\delta(x)+\partial_{x^\alpha}h^{\alpha\beta}(x)\right)D_{x^\beta}+i F^\beta(x)h^\alpha_\beta(x) D_{x^\alpha}.
	\end{align*}
	We will need the following modification of Proposition 1.1 in \cite{LeeUhlmann89}
	\begin{proposition}
		There exists a pseudo-differential operator $A_F(x,D_{x'})\in \Psi_{cl}^1(\partial M)$ which depends smoothly on $x^3$ such that 
		\begin{align*}
		-\Delta_g-g_x(F,\nabla_g\cdot) = \left(D_{x^3}+i\widetilde{E}(x)-iA_F(x,D_{x'})\right)\left(D_{x^3}+iA_F(x,D_{x'})\right),
		\end{align*}
		modulo a smoothing operator.
	\end{proposition}
	
	\begin{proof}
		We construct an asymptotic series for the symbol of $A_F(x,D_{x'})$ using a homogeneity argument. The proposition can be re-stated as the construction of some pseudo-differential operator $A_F(x,D_{x'})$ modulo $\Psi^{-\infty}$ which satisfies the following statement 
		\begin{align*}
		0=\Delta_g+g_x(F,\nabla_g\cdot) + \left(D_{x^3}+i\widetilde{E}(x)-iA_F(x,D_{x'})\right)\left(D_{x^3}+iA_F(x,D_{x'})\right) 
		\end{align*}
		Due to decomposition \eqref{Decomposition for Laplacian}, the problem becomes the construction of a classical, first order pseudo-differential operator $A_F(x,D_{x'})$ which satisfies the following operator equation
		\begin{align*}
		A_F^2-\widetilde{Q}+i[D_{x^3},A_F]-\widetilde{E}A_F = 0
		\end{align*}
		modulo a smoothing operator.
		
		Reduction of the above operator equation to the pseudo-differential symbol calculus yields the following equation (modulo $S^{-\infty}$).
		\begin{equation}
		\sum_{|\mu|\geq 0} \frac{1}{\mu!} \partial_\xi^\mu a D_x^\mu a-\widetilde{q} + \partial_{x^3}a-\widetilde{E}a = 0 \label{Symbol equation}
		\end{equation}
		where $a$ is the full symbol of $A_X$ and $\tilde{q}$ is the full symbol of $\tilde{Q}$ given by 
		\begin{align*}
		\widetilde{q}(x,\xi') &= \sum_{\alpha,\beta}h^{\alpha\beta}(x)\xi_\alpha\xi_\beta -i \sum_{\alpha,\beta} \left(\frac{1}{2}h^{\alpha\beta}(x)\partial_{x^\alpha}\log\delta(x)+\partial_{x^\alpha}h^{\alpha\beta}(x) - F^\alpha(x)h^\beta_\alpha(x) \right)\xi_\beta \\ 
		&= : \widetilde{q}_2(x,\xi')+\widetilde{q}_1(x,\xi')
		\end{align*}
		Let us write
		\begin{align*}
		a(x,\xi') \sim \sum_{j\leq 1} a_j(x,\xi') 
		\end{align*}
		where $a_{j}\in S^{j}_{1,0}(T^*\partial M)$, and is homogeneous of degree $j$ in $\xi'$.
		Collecting terms which are homogeneous of degree $2$ in \eqref{Symbol equation} yields the following 
		\begin{align*}
		a_1^2 - \widetilde{q}_2 = 0 \implies a_1 = \pm \sqrt{\widetilde{q}_2}
		\end{align*}
		For consistency, we will choose $a_1:= -\sqrt{\widetilde{q}_2}$. Next, we will collect the terms which are homogeneous of degree 1 in \eqref{Symbol equation} as follows 
		\begin{align*}
		2a_0a_1+\sum_{\alpha} \partial_{\xi_\alpha}a_1 D_{x^\alpha}a_1-\widetilde{q_1}+\partial_{x^3} a_1-\widetilde{E}a_1 = 0
		\end{align*}
		Solving for $a_0$, we have that 
		\begin{align*}
		a_0 &= -\frac{1}{2a_1} \left(\sum_\alpha \partial_{\xi_\alpha} a_1D_{x^\alpha}a_1 - \widetilde{q_1} + \partial_{x^3}a_1 -\widetilde{E}a_1\right)\\
		&= \frac{1}{2\sqrt{\widetilde{q_2}}}\left(\sum_\alpha \partial_{\xi_\alpha}\sqrt{\widetilde{q_2}}D_{x^\alpha} \sqrt{\widetilde{q_2}} - \widetilde{q_1}- \partial_{x^3} \sqrt{\widetilde{q_2}}+\widetilde{E}\sqrt{\widetilde{q_2}}\right)
		\end{align*}
		We can apply the same recursive argument indefinitely for all degrees of homogeneity $1-j$ in order to obtain $a_{1-j}$ for every $j\geq 1$. For our purposes, the construction of $a_0$ is sufficient. This completes the proof. 
	\end{proof}
	Proposition 1.2 of \cite{LeeUhlmann89} yields the required calculation for the Dirichlet to Neumann operator, given by 
	\begin{align*}
	\Lambda_Ff = \left.\delta^{1/2}A_Ff dx^1\wedge dx^{2}\right\vert_{\partial M} \mod \Omega^2(\partial M) 
	\end{align*}
	Via a one-to-one correspondence, we can associate to this $n-1$ differential form a symbol, denoted by $b(x,\xi')\in S^1_{cl}(T^*\partial M)$ given by 
	\begin{equation}
	b(x,\xi') = \delta^{1/2}a(x,\xi') \label{Symbol for DtN}.
	\end{equation}

\subsection{Explicit Calculation for the Neumann Green's function.}
In this subsection, we will extract the singular part of $G_{\partial M}$ on the diagonal. 

Since $\Lambda_F$ is elliptic and \eqref{left parametrix}, we construct $G_{\partial M}$ as a standard left parametrix of order $-1$. Let $p(x,\xi')\in S^{-1}_{cl}(\partial M)$ be its symbol with the following asymptotic expansion
	\begin{align*}
	p(x,\xi') \sim \sum_{j\geq 1} p_{-j}(x,\xi'),
	\end{align*}
	where $p_{-j}\in S^{-j}(T^*\partial M)$ for each $j\geq 1$. From \eqref{left parametrix}, we deduce that  
	\begin{align*}
	1 
	&= (p\#b)(x,\xi')+S^{-\infty}(T^*\partial M) \\ 
	&= \sum_{|\mu|\geq 0} \frac{1}{\mu!} \partial_{\xi'}^\mu p D_x^\mu \delta^{1/2}a + S^{-\infty}(T^*\partial M),
	\end{align*}
	where $a$ is the symbol constructed in Section \ref{SND}. By matching the terms with the same orders, we obtain
	\begin{align*}
	p_{-1}(x,\xi') = \frac{\chi(\xi')}{\delta^{1/2}\sqrt{\widetilde{q_2}(x,\xi')}}
	\end{align*}
	\begin{align*}
	p_{-2}(x,\xi') = \frac{\chi(\xi')}{\delta^{1/2}\sqrt{\widetilde{q_2}(x,\xi')}}\left(\frac{\chi(\xi')a_0(x,\xi')}{\sqrt{\widetilde{q_2}(x,\xi')}}-\sum_{|\mu|=1} \partial_{\xi'}^\mu\frac{\chi(\xi')}{\sqrt{\widetilde{q_2}(x,\xi')}}D_x^\mu\delta^{1/2}\sqrt{\widetilde{q_2}(x,\xi')}  \right).
	\end{align*}
	Here, $\chi$ is a smooth cutoff function, non-zero outside of some sufficiently large neighbourhood of the origin. The choice of terms $p_{-j}$ for $j\geq 3$ can be done via standard, iterative parametrix arguments which were used to obtain $p_{-1}$ and $p_{-2}$. For the sake of brevity, such computations shall be omitted and are unnecessary for our purposes. It should be noticed that for $x=x^*$, the origin of our geodesic disk, as a result of boundary normal co-ordinates, the symbol terms are reduce to
	\begin{align*}
	    p_{-1}(x^*,\xi') = \frac{\chi(\xi')}{|\xi'|}.
	\end{align*}
	In addition, we have that 
	\begin{align*}
	    p_{-2}(x^*,\xi') = \frac{\chi(\xi')}{|\xi'|}\left(\frac{\chi(\xi')a_0(x^*,\xi')}{|\xi'|}-\left.\sum_{|\mu|=1}\partial_{\xi'}^\mu \frac{\chi(\xi')}{\sqrt{\widetilde{q_2}(x,\xi')}}D_x^\mu \sqrt{\widetilde{q_2}(x,\xi')}\right\vert_{x=x^*}\right).
	\end{align*}
	Notice that in $p_{-2}$, the $\delta^{1/2}$ term is peeled off by the product rule and vanishes as a result of $\partial_{x^k} h_{ij}=0$ at the origin. 

	The calculation of the principle symbol for $G_{\partial M}$ as well as the the next highest order term yield the following asymptotic for the kernel centred at $x^*$, evaluated at $x\in B_h(\rho; x^*)\subset \partial M$
	\begin{align*}
    G_{\partial M}(x^*,x) &= 
    \Phi^*G_{\partial M}(0,t') \\&= \frac{1}{4\pi^2}\int_{\mathbb{R}^2}e^{-i\xi'\cdot t'} p_{-1}(x^*,\xi') d\xi' + \frac{1}{4\pi^2}\int_{\mathbb{R}^2} e^{-i\xi'\cdot t'}p_{-2}(x^*,\xi') d\xi'+ \Psi_{cl}^{-3}(\partial M).
\end{align*}
The first term evaluates to 
\begin{align*}
    \frac{1}{4\pi^2}\int_{\mathbb{R}^2}e^{-i\xi'\cdot t'} p_{-1}(x^*,\xi') d\xi' = \frac{1}{4\pi^2\delta^{1/2}}\int_{\mathbb{R}^2}e^{-i\xi'\cdot t'} \frac{\chi(\xi')}{|\xi'|}d\xi'.
\end{align*}
Furthermore, we can split the above integral into a singular and regular part as follows
\begin{align*}
   \frac{1}{4\pi^2}\int_{\mathbb{R}^2}e^{-i\xi'\cdot t'} \frac{1}{|\xi'|}d\xi' + I_{\text{reg}}.
\end{align*}
The regular part is of no interest to us and will be lost in the error as this computation is done in order to isolate the most prevalent singularities occurring in $G_{\partial M}$. This idea can be carried forth in the computation for the second integral. We have that 
\begin{align*}
    \frac{1}{4\pi^2}\int_{\mathbb{R}^2}e^{-i\xi'\cdot t'} \frac{1}{|\xi'|}d\xi' &= \mathcal{F}(|\xi'|^{-1})(t') = \frac{1}{2\pi}|t'|^{-1}.
\end{align*}
Next, we use the previous arguments and proceed to split the second term up as follows
\begin{align*}
    \frac{1}{4\pi^2}\int_{\mathbb{R}^2}e^{-i\xi'\cdot t'} &\frac{a_0(0,\xi')}{|\xi'|^2}d\xi'\\
    & - \sum_{|\mu|=1}\frac{1}{4\pi^2}\int_{\mathbb{R}^2}e^{-i\xi'\cdot t'} \frac{1}{|\xi'|} \left.\partial_{\xi'}^\mu (\widetilde{q_2}(x,\xi'))^{-1/2}\right\vert_{x=x^*}\left. D_x^\mu \sqrt{\widetilde{q_2}(x,\xi')}\right\vert_{x=x^*} d\xi'.
\end{align*}
We recall that we have that $\widetilde{q_2}(x,\xi') = h^{\alpha\beta}(x)\xi_\alpha\xi_\beta$. Thus, we have the following identity 
\begin{align*}
    D_x^\mu \sqrt{h^{\alpha\beta}(x)\xi_\alpha\xi_\beta} = \frac{D_x^\mu h^{\alpha\beta}(x)}{2q_1(x,\xi')}.
\end{align*}
In boundary normal co-ordinates, centred at $x=x^*$, this identity evaluates to $0$. Thus, we have that the second integral is vanishing. We are now left to compute the following integral
\begin{align*}
    I(t'):= \frac{1}{8\pi^2} \left.\int_{\mathbb{R}^2} e^{-i\xi'\cdot t'}\frac{1}{|\xi'|^3}\left(\sum_\alpha \partial_{\xi_\alpha}\sqrt{\widetilde{q_2}}D_{x^\alpha} \sqrt{\widetilde{q_2}} - \widetilde{q_1}- \partial_{x^3} \sqrt{\widetilde{q_2}}+\widetilde{E}\sqrt{\widetilde{q_2}}\right)\right\vert_{x=x^*} d\xi'.
\end{align*}
Furthermore, since $D_{x^\alpha}\sqrt{\widetilde{q_2}(x,\xi')}=0$ and $\widetilde{q_1}(x,\xi')=iX^\alpha(x)\xi_\alpha$ in boundary normal co-ordinates, we have that $I$ is given by 
\begin{align*}
    I(t') &= \frac{1}{8\pi^2} \int_{\mathbb{R}^2} e^{-i\xi'\cdot t'}\frac{1}{|\xi'|^3}\left.\left(\widetilde{E}|\xi'|-\partial_{x^3}\sqrt{h^{\alpha\beta}(x)\xi_\alpha\xi_\beta} -i F^\alpha(x)\xi_\alpha\right)\right\vert_{x=x^*}d\xi' \\ 
    &= \frac{\widetilde{E}(x^*)}{8\pi^2} \int_{\mathbb{R}^2} e^{-i\xi'\cdot t'} \frac{1}{|\xi'|^2} d\xi' - \frac{1}{8\pi^2} \int_{\mathbb{R}^2} e^{-i\xi'\cdot t'} \left.\frac{\partial_{x^3}\sqrt{h^{\alpha\beta}(x)\xi_\alpha\xi_\beta}}{|\xi'|^3} \right\vert_{x=x^*}d\xi'\\
    &- \frac{i}{8\pi^2} F^\alpha(x^*)\int_{\mathbb{R}^2} e^{-i\xi'\cdot t'} \frac{\xi_\alpha}{|\xi'|^3}d\xi'.
\end{align*}
We calculate each term in the integral in order of increasing difficulty. Using proposition 8.17 [Lee] we can ascertain the following
\begin{align*}
    \widetilde{E}(x^*) = \left.-\frac{1}{2}\sum_{\alpha,\beta} h^{\alpha\beta}(x)\partial_{x^3}h_{\alpha\beta}(x)\right\vert_{x=x^*}-F^3(x^*) = 2H(x^*)-F^3(x^*).
\end{align*}
Where $H(x^*)$ denotes the mean curvature of $(M,g,\partial M)$ at $x^*$. Thus, we have that
\begin{align*}
    \frac{\widetilde{E}(x^*)}{8\pi^2} \int_{\mathbb{R}^2} e^{-i\xi'\cdot t'} \frac{1}{|\xi'|^2} d\xi' = \frac{2H(x^*)-F^3(x^*)}{2}\mathcal{F}(|\xi'|^{-2}) = -\frac{2H(x^*)-F^3(x^*)}{4\pi}\log|t'|.
\end{align*}
Next, we calculate the third term by making the following observation
\begin{align*}
    \partial_{\xi_\alpha} |\xi'|^{-1} = -\frac{\xi_\alpha}{|\xi'|^3}.
\end{align*}
We can re-write the third term as follows 
\begin{align*}
    -\frac{i}{8\pi^2}\sum_\alpha F^\alpha(x^*) \int_{\mathbb{R}^2} e^{-i\xi'\cdot t'} \frac{\xi_\alpha}{|\xi'|^3}d\xi' &= \frac{i}{8\pi^2} \sum_\alpha F^\alpha(x^*)\int_{\mathbb{R}^2} e^{-i\xi'\cdot t'} \partial_{\xi_\alpha}|\xi'|^{-1} d\xi' \\ 
    &= \frac{1}{8\pi^2}\sum_\alpha t_\alpha F^\alpha(x^*) \int_{\mathbb{R}^2} e^{-i\xi'\cdot t'}  |\xi'|^{-1} d\xi'\\ 
    &= \frac{1}{2}\sum_\alpha t_\alpha F^\alpha(x^*)\mathcal{F}(|\xi'|^{-1})(t') \\ 
    &= \sum_\alpha F^\alpha(x^*) \frac{t_\alpha}{4\pi|t'|}.
\end{align*}
Lastly, we compute the second term. It should be noted that via rotation with basis $E_1^\flat \wedge E_2^\flat \wedge \nu^\flat$, we find that the scalar second fundamental form is diagonalised in our co-ordinate system. Thus, we have the following
\begin{align*}
    &-\frac{1}{8\pi^2} \int_{\mathbb{R}^2}e^{-i\xi'\cdot t'} \frac{\partial_{x^3} |\xi'|_h}{|\xi'|^3} d\xi'\\
    &= -\frac{1}{8\pi^2}\int_{\mathbb{R}^2} e^{-i\xi'\cdot t'} \frac{\left.\left(\partial_{x^3}h^{11}(x) + 2\partial_{x^3}h^{12}(x)+\partial_{x^3}h^{22}(x)\right)\right\vert_{x=x^*}}{|\xi'|^4} d\xi' \\
    &= -\frac{1}{8\pi^2} \int_{\mathbb{R}^2} e^{-i\xi'\cdot t'} \frac{\lambda_1(x^*)\xi_1^2 +\lambda_2(x^*) \xi_2^2}{|\xi'|^4}d\xi' \\ 
    &= \frac{1}{4\pi^2}\left(\frac{\lambda_1(x^*)+\lambda_2(x^*)}{2}\pi\log|t'|-\frac{\pi}{4}\left(\frac{\lambda_2(x^*)t_1^2+\lambda_1(x^*)t_2^2}{|t'|^2}-\frac{\lambda_1(x^*)t_1^2+\lambda_2(x^*)t_2^2}{|t'|^2}\right)\right)\\
\end{align*}
Here $\lambda_1,\lambda_2$ denote the associated principle curvatures. Therefore, we have that $I(t')$ is given by
\begin{align*}
    I(t') = &-\frac{2H(x^*)-F^3(x^*)}{4\pi}\log|t'|\\ &+\frac{1}{4\pi^2}\left(H(x^*)\pi\log|t'|-\frac{\pi}{4}\left(\frac{\lambda_2(x^*)t_1^2+\lambda_1(x^*)t_2^2}{|t'|^2}-\frac{\lambda_1(x^*)t_1^2+\lambda_2(x^*)t_2^2}{|t'|^2}\right)\right)\\
    &+\frac{1}{4\pi} \sum_\alpha \frac{F^\alpha(x^*) t_\alpha}{|t'|}
\end{align*}
Therefore, we have that 
\begin{align*}
    I(\exp^{-1}_{x^*}(x)) =& -\frac{H(x^*)}{4\pi}\log d_h(x^*,x) + \frac{F^3(x^*)}{4\pi}\log d_h(x^*,x) \\
    &+ \frac{1}{16\pi}\left(\text{II}_{x^*}\left(\frac{\exp_{x^*}^{-1}(x)}{|\exp_{x^*}^{-1}(x)|_h}\right)-\text{II}_{x^*}\left(\frac{\star\exp_{x^*}^{-1}(x)}{|\exp_{x^*}^{-1}(x)|_h}\right)\right)\\
    &+ \frac{1}{4\pi}h_{x^*}\left(F^{\parallel}(x^*),\frac{\exp_{x^*}^{-1}(x)}{|\exp_{x^*}^{-1}(x)|_h}\right).
\end{align*}
This yields the following asymptotic for $G_{\partial M}$ 
\begin{align*}
    G_{\partial M}(x^*,x) &= \frac{1}{2\pi}d_h(x^*,x)^{-1} -\frac{H(x^*)}{4\pi}\log d_h(x^*,x)+\frac{X^3_{x^*}}{4\pi}\log d_h(x^*,x)\\ &+\frac{1}{16\pi}\left(\text{II}_{x^*}\left(\frac{\exp_{x^*}^{-1}(x)}{|\exp_{x^*}^{-1}(x)|_h}\right)-\text{II}_{x^*}\left(\frac{\star\exp_{x^*}^{-1}(x)}{|\exp_{x^*}^{-1}(x)|_h}\right)\right)\\
    &+\frac{1}{4\pi}h_{x^*}\left(F^{\parallel}(x^*),\frac{\exp_{x^*}^{-1}(x)}{|\exp_{x^*}^{-1}(x)|_h}\right)  +\Psi^{-3}(\partial M).
\end{align*}
We can make a further refinement on the above series by invoking Corollary 2.5 from \cite{NTT} in order to write the following
\begin{align*}
    G_{\partial M}(x^*,x) &= \frac{1}{2\pi}d_g(x^*,x)^{-1} -\frac{H(x^*)}{4\pi}\log d_h(x^*,x)+\frac{F^3(x^*)}{4\pi}\log d_h(x^*,x)\\ &+\frac{1}{16\pi}\left(\text{II}_{x^*}\left(\frac{\exp_{x^*}^{-1}(x)}{|\exp_{x^*}^{-1}(x)|_h}\right)-\text{II}_{x^*}\left(\frac{\star\exp_{x^*}^{-1}(x)}{|\exp_{x^*}^{-1}(x)|_h}\right)\right)\\
    &+\frac{1}{4\pi}h_{x^*}\left(F^{\parallel}(x^*),\frac{\exp_{x^*}^{-1}(x)}{|\exp_{x^*}^{-1}(x)|_h}\right)  +\Psi^{-3}(\partial M).
\end{align*}	
The above work yields an expression for $G_{\partial M}$ centred at the origin of the window $x^*\in \Gamma_{\varepsilon,a}$. We however need an expansion for $G_{\partial M}(x,y)$ for $x,y$ close in $\Gamma_{\varepsilon,a}$. In particular, we need an expression given by \eqref{res coord}. In order to do so, we simply invoke the following proposition, which was proven in \cite{NTT}.
	\begin{proposition}
		\label{rescaled kernel in loc coord}
		Let $x_0\in \partial M$ and $\lambda_1(x_0)$ and $\lambda_2(x_0)$ be the eigenvalues of the shape operator at $x_0$. Assume that $x=x^\varepsilon(s';x_0)$, $y=x^\varepsilon(t';x_0)$, $r = |s'-t'|$ and $t' = s+ r\omega$. Then, for $\varepsilon>0$ sufficiently small, we have that
		\begin{eqnarray*} 
			d_g(x,y)^{-1} = \varepsilon^{-1}r^{-1} + \varepsilon r^{-1} A^1(\varepsilon, s', r,\omega),
		\end{eqnarray*}
		\begin{eqnarray*} 
			d_h(x,y)^{-1} = \varepsilon^{-1}r^{-1} + \varepsilon r^{-1} A^2(\varepsilon, s', r,\omega),
		\end{eqnarray*}
		$$ 
		\II_x \left(\frac{\e_ {x;h}^{-1}y}{|\e_{x;h}^{-1} y|_h}, \frac{\e _{x;h}^{-1}y}{|\e _{x,h}^{-1} y|_h}\right) = \left( \lambda_1(x_0) \frac{(s_1 - t_1)^2}{r^2} + \lambda_2(x_0) \frac{(s_2 - t_2)^2}{r^2}\right) + \varepsilon R^1_\varepsilon (t',\omega, r),
		$$
		$$ 
		\II_x \left(*\frac{\e_ {x;h}^{-1}y}{|\e_{x;h}^{-1} y|_h}, *\frac{\e _{x;h}^{-1}y}{|\e _{x,h}^{-1} y|_h}\right) = \left( \lambda_2(x_0) \frac{(s_1 - t_1)^2}{r^2} + \lambda_1(x_0) \frac{(s_2 - t_2)^2}{r^2}\right) + \varepsilon R^2_\varepsilon (t',\omega, r),
		$$
		\begin{align*}
		h_x\left(F^{\parallel}(x), \frac{\exp_{x; h}(y)}{|\exp_{x; h}(y)|_h}\right) = \frac{F_1(t')(s_1 - t_1) + F_2(t')(s_2 - t_2)}{r} + \varepsilon R_\varepsilon^3(t',\omega,r),
		\end{align*}
		where $A^1$, $A^2$, and $A^3$ are smooth in $[0,\varepsilon_0]\times \D\times \R \times S^1$ and $R^1_\varepsilon$, $R^2_\varepsilon$, $R^3_\varepsilon$ are smooth in $\D \times S^2\times [0,\rho]$ with derivatives of all orders uniformly bounded in $\varepsilon$.
	\end{proposition}
	\begin{proof}
		The first two statements were proved in Corollaries 2.6 and 2.3 of \cite{NTT}, respectively. The next two results were proved in Corollary 2.9 of \cite{NTT}. The last one follows from Lemma 2.8.
	\end{proof}


\section{Proof of Theorems \ref{main theorem disk} and \ref{main theorem}}\label{pr main rres}
	In this section we give a proof for Theorems \ref{main theorem disk} and \ref{main theorem}. We recall that the mean sojourn time, $u_{\varepsilon,a}$, satisfies the elliptic mixed boundary value problem \eqref{main bvp}, this is proved in the Appendix. Therefore, by using Green's identity, we show that $u_{\varepsilon,a}$ time satisfies the integral equation	
	\begin{align}\label{org_int_eq_pr main rres}
		u_{\varepsilon,a}(x) = \mathcal{G}(x) + \int_{\partial M}G(x,z) \partial_{\nu_z} u_{\varepsilon,a}(z) d_h(z) + C_{\varepsilon,a},
	\end{align}
	where $x\in M^0$, $G$ is the Neumann Green function we disscused in Section \ref{Green} and
	\begin{equation*}
	C_{\varepsilon,a} := \frac{1}{|\partial M|} \int_{\partial M} u_{\varepsilon,a}(z) d_h(z), \qquad \mathcal{G}(x) := \int_M G(x,z) d_h(z).
	\end{equation*}
	From \eqref{adj}, it follows that $\mathcal{G}$ sutisfies
	\begin{equation}\label{eq for mathcal G}
	\begin{cases}
	\Delta_{g} \mathcal{G}(x) + g ( F(x), \nabla_{g}\mathcal{G}(x) ) = -1;\\
	\partial_{\nu}\mathcal{G}(x)\left.\right|_{x\in\partial M} = -\frac{\Phi(x)}{|\partial M|};\\
	\int_{\partial M} \mathcal{G}(x) d_h(x) =0.
	\end{cases}
	\end{equation}
	By taking the trace of the integral equation \eqref{org_int_eq_pr main rres} to $x\in \Gamma_{\varepsilon, a}$, we obtain
	\begin{equation*}
		0 = \mathcal{G}(x) + \int_{\partial M}G_{\partial M}(x,z) \partial_{\nu_z} u_{\varepsilon,a}(z) d_h(z) + C_{\varepsilon,a}.
	\end{equation*}
	Therefore, Proposition \ref{G exp sec 3} gives
	\begin{align}\label{int_eq_mani}
	\nonumber-\mathcal{G}(x)-C_{\varepsilon,a} &= \frac{1}{2\pi}\int_{\Gamma_{\varepsilon,a}} d_g(x,y)^{-1}\partial_\nu u_{\varepsilon,a}(y)d_h(y)\\
	\nonumber&-\frac{H(x)-\partial_{\nu}\phi(x)}{4\pi}\int_{\Gamma_{\varepsilon,a}} \log d_h(x,y) \partial_\nu u_{\varepsilon,a}(y)d_h(y)\\
	&+\frac{1}{16\pi}\int_{\Gamma_{\varepsilon,a}}\left(\text{II}_x\left(\frac{\exp_x^{-1}(y)}{|\exp_x^{-1}(y)|_h}\right)-\text{II}_x\left(\frac{\star\exp_x^{-1}(y)}{|\exp_x^{-1}(y)|_h}\right)\right)\partial_\nu u_{\varepsilon,a}(y)d_h(y)  \\
	\nonumber&+\frac{1}{4\pi}h_x\left(F^{\parallel}(x), \frac{\exp_{x; h}(y)}{|\exp_{x; h}(y)|_h}\right) \partial_\nu u_{\varepsilon,a}(y)d_h(y)\\
	\nonumber& + \int_{\Gamma_{\varepsilon,a}} R(x,y)\partial_\nu u_{\varepsilon,a}(y)d_h(y).
	\end{align}
	Since $F=\nabla_g \phi$, the fact that $u_{\varepsilon,a}$ satisfies \eqref{main bvp} implies
	\begin{equation*}
	\textrm{div}_g(e^\phi \nabla_{g} u_{\varepsilon,a}) = e^\phi (\Delta_g u_{\varepsilon,a }+F\cdot \nabla_g u_{\varepsilon,a}) = - e^\phi.
	\end{equation*}
	By the divergence form theorem, we know that 
	\begin{equation*}
	\int_{M} \textrm{div}_g(e^\phi \nabla_{g} u_{\varepsilon,a})(z) d_g(z) = \int_{\partial M} e^\phi \partial_{\nu}u_{\varepsilon,a}(z) d_h(z).
	\end{equation*}
	Thereofore, by integrating the penultimate equation, we derive the following compatibility condition
	\begin{equation}\label{compatibility_condition}
	\int_{\partial M} e^\phi \partial_{\nu}u_{\varepsilon,a}(z) d_h(z) = - \int_{M} e^{\phi(z)} d_g(z).
	\end{equation}
	
	We will use the coordinate system given by 
	\begin{eqnarray}\label{ellipse coord} \D \ni (s_1, s_2) \mapsto x^\varepsilon(s_1, as_2; x^*) \in \Gamma_{\varepsilon, a},
	\end{eqnarray}
	where $ x^\varepsilon (\cdot; x^*): {\mathcal E}_a \to \Gamma_{\varepsilon, a}$ is the coordinate defined in Section \eqref{notations}. To simplify notation we will drop the $x^*$ in the notation and denote $x^\varepsilon(\cdot; x^*)$ by simply $x^\varepsilon(\cdot)$.
	
	Note that in these coordinates the volume form for $\partial M$ is given by 
	\begin{eqnarray}
	\label{volume form in epsilon}
	d_{h} = a \varepsilon^2(1 + \varepsilon^2 Q_\varepsilon(s') )ds_1 \wedge ds_2,\ s'\in \D
	\end{eqnarray}
	for some smooth function $Q_\varepsilon (s')$ whose derivatives of all orders are bounded uniformly in $\varepsilon$. We denote 
	\begin{eqnarray}
	\label{def of psi}
	\psi_\varepsilon(s') := \partial_\nu u_\varepsilon (x^\varepsilon(s_1, as_2)).
	\end{eqnarray}
	Then, in this coordinate system, the compatibility condition become
	\begin{equation}\label{compatibility_condition_loc}
	\int_{\mathbb{D}} e^{\phi(x^{\varepsilon}(s'))} \psi_{\varepsilon}(s') (1 + \varepsilon^2 Q_\varepsilon(s') ) ds' = -\frac{1}{a\varepsilon^2} \int_{M} e^{\phi(z)} d_g(z).
	\end{equation}
	We can re-write this as follow 
	\begin{equation*}
	e^{\phi(x^*)}\int_{\D} \psi(x)dx + \int_{\D} \left[(e^{\phi(x^{\varepsilon}(s'))} - e^{\phi(x^*)}) + \varepsilon^2 e^{\phi(x^*)}Q_{\varepsilon}(s')\right]\psi_{\varepsilon}(s')ds'= -\frac{1}{a\varepsilon^2} \int_{M} e^{\phi(z)} d_g(z).
	\end{equation*}
	Since $\phi$ is smooth, we conclude that
	\begin{equation}\label{int psi}
	\int_{\D} \psi(x)dx = -\frac{\Phi(x^*)}{a\varepsilon^2} + \varepsilon \int_{\D} \tilde{Q}_{\varepsilon}(s') \psi_{\varepsilon}(s')ds',
	\end{equation}
	where $\Phi$ is the function defined in Theorem \eqref{main theorem disk} and $\tilde{Q}_\varepsilon (s')$
	is some smooth function whose derivatives of all orders are bounded uniformly in $\varepsilon$.
	
	Next, we re-write \eqref{int_eq_mani} in this coordinate system given by \eqref{ellipse coord}. To do this, let us first introduce the following operators. Consider
	\begin{equation}
	\label{La}
	L_a f = a \int_{\mathbb{D}} \frac{f(s')}{\left((t_1 - s_1)^2 + a^2(t_2 - s_2)^2\right)^{1/2}} ds'
	\end{equation}
	acting on functions of the disk $\D$. By \cite{schuss2006ellipse} we have that
	\begin{equation}
	\label{La u = 1}
	L_a  \left({K_a}^{-1} {(1 - |t'|^2)^{-1/2}}\right) = 1,
	\end{equation}
	on $\D$ where 
	\begin{equation*}
	K_a = \frac{\pi}{2} \int_{0}^{2 \pi} \frac{1}{\left( \cos^2 \theta + \frac{\sin^2 \theta}{a^2} \right)^{1/2}} d \theta. 
	\end{equation*}
	By (4.4) in \cite{NTT}, this is the unique solution in $H^{1/2}(\D)^*$ to $L_a u = 1$.
	
	Next we denote
	\begin{equation*}
	R_{\log,a} f(t') : = a \int_\D \log\left((t_1 - s_1)^2 + a^2 (t_2-s_2)^2 \right)^{1/2} f(s') ds',
	\end{equation*}
	
	\begin{equation*}
	R_{\infty,a} f(t') : = a \int_{\mathbb{D}} \frac{(t_1 - s_1)^2 - a^2 (t_2 - s_2)^2}{(t_1 - s_1)^2 + a^2 (t_2 - s_2)^2} f(s') ds',
	\end{equation*}
	
	\begin{equation*}
	R_{F,a} f(t') : = a \int_{\mathbb{D}} \frac{F^1(0)(t_1 - s_1) + a F^2(0) (t_2 - s_2)}{((t_1 - s_1)^2 + a^2 (t_2 - s_2)^2)^{1/2}} f(s') ds'.
	\end{equation*}
	\begin{remark}\label{rem}
		In \cite{NTT}, it was showed that the operators $R_{\log,a}$ and $R_{\log,a}$ are bounded maps from $H^{1/2}(\D)^*$ to $H^{3/2}(\D)$. By repeating the arguments, one can show that this is also true for $R_{F,a}$.
	\end{remark}

	We unwrap the right hand side of \eqref{int_eq_mani} term by term in the following five lemmas. Note that the first four of them are proved in \cite{NTT}. We repeat them here for the convenience of the readers.
	
	\begin{lemma}\label{l1}
		We have the following identity
		\begin{equation*}
			\int_{\Gamma_{\varepsilon,a}} d_g(x,y)^{-1}\partial_\nu u_{\varepsilon,a}(y)d_h(y) = \varepsilon L_a\psi_\varepsilon(t') + \varepsilon^3 \mathcal{A}_\varepsilon\psi_\varepsilon(t')
		\end{equation*}
		with $x=x^{\varepsilon}(t')$ for some $\A_\varepsilon : H^{1/2}(\D; ds')^* \to H^{1/2}(\D; ds')$ with operator norm bounded uniformly in $\varepsilon$.
	\end{lemma}

	\begin{proof}
		By using Proposition \ref{rescaled kernel in loc coord} and \eqref{volume form in epsilon}, we see that in the coordinate system \eqref{ellipse coord} it follows
		\begin{align*}
			\int_{\Gamma_{\varepsilon,a}} d_g(x,y)^{-1}\partial_\nu u_{\varepsilon,a}(y)d_h(y)= &a\varepsilon \int_{\D} \frac{1}{((s_1-t_1)^{2} + a (s_2-t_2)^{2})^{1/2}}\psi_{\varepsilon}(s')ds'\\
			&+ a\varepsilon^3 \int_{\D} \frac{A^1(s',t')(1 + \varepsilon^2 Q_\varepsilon(s') ) + Q_\varepsilon(s')}{((s_1-t_1)^{2} + a (s_2-t_2)^{2})^{1/2}}\psi_{\varepsilon}(s')ds'.\\
		\end{align*}
		Therefore, by Lemma 2.11 of \cite{NTT}, the second term of the right-hand side can be written as $\varepsilon^3\mathcal{A}_{\varepsilon}\psi$, for operator $\A_\varepsilon$ which satisfies the requirement of the statement.
	\end{proof}
	From now on, we will denote by $\A_\varepsilon$ any operator which takes $H^{1/2}(\D ;ds')^* \to H^{1/2}(\D; ds')$ whose operator norm is bounded uniformly in $\varepsilon$. 
	
	For the second term of the right-hand side of \eqref{int_eq_mani}, the following lemma holds.
	\begin{lemma}\label{l2}
		We have the following identity
		\begin{align*}
			(H(x)-\partial_{\nu}\phi(x))&\int_{\Gamma_{\varepsilon,a}} \log d_h(x,y) \partial_\nu u_{\varepsilon,a}(y)d_h(y) = -(H(x^*) - \partial_{\nu}\phi(x^*)) \Phi(x^*) \log\varepsilon\\
			&  +\varepsilon^2 (H(x^*) - \partial_{\nu}\phi(x^*)) R_{log,a} \psi_{\varepsilon}(t') + O_{ H^{1/2}(\D)}(\varepsilon\log\varepsilon) + \varepsilon^3\log\varepsilon \mathcal{A}_{\varepsilon}\psi_{\varepsilon}.
		\end{align*}
		where $x=x^{\varepsilon}(t')$ and $\Phi$ is the function defined in Theorem \eqref{main theorem disk}.
	\end{lemma}
	\begin{proof}
		By using Proposition \eqref{rescaled kernel in loc coord} and \eqref{volume form in epsilon}, we obtain
		\begin{align*}
		&(H(x)-\partial_{\nu}\phi(x))\int_{\Gamma_{\varepsilon,a}} \log d_h(x,y) \partial_\nu u_{\varepsilon,a}(y)d_h(y)\\
		&=a \varepsilon^2 \log\varepsilon(H(x^{\varepsilon}(t')) - \partial_{\nu}\phi(x^{\varepsilon}(t')))  \int_{\D} \psi_{\varepsilon}(s') ds'\\
		&\quad + a \varepsilon^2 (H(x^{\varepsilon}(t')) - \partial_{\nu}\phi(x^{\varepsilon}(t'))) \int_{\D} \log \left[((t_1-s_1)^2 + a(t_2-s_2)^2)^{1/2}\right]\psi_{\varepsilon}(s') ds'\\
		&\quad - a \varepsilon^2 (H(x^{\varepsilon}(t')) - \partial_{\nu}\phi(x^{\varepsilon}(t'))) \int_{\D} \log \left( 1 + \varepsilon^2 A(s',t')\right)\psi_{\varepsilon}(s') ds'\\
		&\quad - a \varepsilon^4 (H(x^{\varepsilon}(t')) - \partial_{\nu}\phi(x^{\varepsilon}(t'))) \int_{\D} \log \left[\varepsilon((t_1-s_1)^2 + a(t_2-s_2)^2)^{1/2}\right] Q_{\varepsilon}(s') \psi_{\varepsilon}(s') ds'\\
		&\quad - a \varepsilon^4 (H(x^{\varepsilon}(t')) - \partial_{\nu}\phi(x^{\varepsilon}(t'))) \int_{\D} \log \left( 1 + \varepsilon^2 A(s',t')\right) Q_{\varepsilon}(s') \psi_{\varepsilon}(s') ds'.
		\end{align*}
		Apart the first and second terms, all terms of the right-hand side can be written as $\varepsilon^3\mathcal{A}_{\varepsilon}\psi_{\varepsilon}$. The first term, by \eqref{int psi}, is equals to
		\begin{align*}
			-(H(x^*) - \partial_{\nu}\phi(x^*)) \Phi(x^*)&\log\varepsilon +\varepsilon^3 \log\varepsilon (H(x^{\varepsilon}(t')) - \partial_{\nu}\phi(x^{\varepsilon}(t')))\int_{\D} \tilde{Q}_{\varepsilon}(s') \psi_{\varepsilon}(s')ds'\\
			&+ \left(H(x^*)- \partial_{\nu}\phi(x^*)-H(x^{\varepsilon}(t') + \partial_{\nu}\phi(x^{\varepsilon}(t')))\right) \log\varepsilon \Phi(x^*),
		\end{align*}
		consequently, by Lemma 2.11 of \cite{NTT}, it is equal to 
		\begin{align*}
			-(H(x^*) - \partial_{\nu}\phi(x^*)) \Phi(x^*) \log\varepsilon  + O_{ H^{1/2}(\D)}(\varepsilon\log\varepsilon) + \varepsilon^3\log\varepsilon \mathcal{A}_{\varepsilon}\psi_{\varepsilon}.
		\end{align*}
		While the second term is
		\begin{align*}
			\varepsilon^2 (H(x^*) - \partial_{\nu}\phi(x^*)) &R_{log,a} \psi_{\varepsilon}(t')\\
			&+ \varepsilon^2 (H(x^{\varepsilon}(t')) - \partial_{\nu}\phi(x^{\varepsilon}(t')) - H(x^*) + \partial_{\nu}\phi(x^*)) R_{log,a} \psi_{\varepsilon}(t').
		\end{align*}
		Since $H$ and $\phi$ are smooth functions, we derive that the last term of the above expression is $\varepsilon^3\mathcal{A}_\varepsilon\psi_{\varepsilon}$.
	\end{proof}
	For the forth term of  \eqref{int_eq_mani}, we have the following lemma.
	\begin{lemma}\label{l3}
		We have the following identity
		\begin{align*}
			\int_{\Gamma_{\varepsilon,a}}\left(\text{II}_x\left(\frac{\exp_x^{-1}(y)}{|\exp_x^{-1}(y)|_h}\right)-\text{II}_x\left(\frac{\star\exp_x^{-1}(y)}{|\exp_x^{-1}(y)|_h}\right)\right)\partial_\nu u_{\varepsilon,a}(y)d_h(y)\\
			=\varepsilon^2(\lambda_1(x^*)-\lambda_2(x^*)) R_{\infty,a} \psi_\varepsilon(t') + \varepsilon^3 \A_\varepsilon \psi_\varepsilon.
		\end{align*}
	\end{lemma}
		\begin{proof}
			By Proposition \ref{rescaled kernel in loc coord} and \eqref{volume form in epsilon}, we see that 
			\begin{align*}
				\int_{\Gamma_{\varepsilon,a}}&\left(\text{II}_x\left(\frac{\exp_x^{-1}(y)}{|\exp_x^{-1}(y)|_h}\right)-\text{II}_x\left(\frac{\star\exp_x^{-1}(y)}{|\exp_x^{-1}(y)|_h}\right)\right)\partial_\nu u_{\varepsilon,a}(y)d_h(y)\\
				& = a(\lambda_1(x^*) - \lambda_2(x^*))\varepsilon^2\int_{\D} \frac{(s_1 - t_1)^2 - a^2(s_2 - t_2)^2}{(s_1 - t_1)^2 + a^2(s_2 - t_2)^2} \psi_{\varepsilon}(s') ds'\\
				& \quad +a \varepsilon^3 \int_{\D} \left(R^1(t',s') - R^2(t',s')\right)(1 + \varepsilon^2 Q(s'))\psi_{\varepsilon}(s')ds'\\
				& \quad + a(\lambda_1(x^*) - \lambda_2(x^*)) \varepsilon^4 \int_{\D} \frac{(s_1 - t_1)^2 - a^2(s_2 - t_2)^2}{(s_1 - t_1)^2 + a^2(s_2 - t_2)^2} Q(s') \psi_{\varepsilon}(s') ds'.
			\end{align*}
			We use Lemma 2.11 of \cite{NTT} to complete the proof.
		\end{proof}
	Next, we stady the forth term of \eqref{int_eq_mani}.
	\begin{lemma}\label{l4}
		The folloing is true
		\begin{align*}
			\int_{\Gamma_{\varepsilon,a}} h_x\left(F^{\parallel}(x), \frac{\exp_{x; h}(y)}{|\exp_{x; h}(y)|_h}\right) \partial_\nu u_{\varepsilon,a}(y)d_h(y) = \varepsilon ^2 R_{F,a}\psi_{\varepsilon} + \varepsilon^3\mathcal{A}_\varepsilon\psi_{\varepsilon}.
		\end{align*}
	\end{lemma}
	\begin{proof}
		By Proposition \ref{rescaled kernel in loc coord} and \eqref{volume form in epsilon}, we get
		\begin{align*}
		\int_{\Gamma_{\varepsilon,a}} h_x&\left(F^{\parallel}(x), \frac{\exp_{x; h}(y)}{|\exp_{x; h}(y)|_h}\right) \partial_\nu u_{\varepsilon,a}(y)d_h(y) \\
		&= a \varepsilon^2 \int_{\D} \frac{F_1(t')(s_1 - t_1) + F_2(t')(s_2 - t_2)}{((s_1 - t_1)^2 + a^2(s_2 - t_2)^2)^{1/2}} \psi_{\varepsilon}(s') ds'\\
		& \quad +a \varepsilon^3 \int_{\D} R^3(t',s')(1 + \varepsilon^2 Q(s'))\psi_{\varepsilon}(s')ds'\\
		& \quad + a\varepsilon^4 \int_{\D} \frac{F_1(t')(s_1 - t_1) + F_2(t')(s_2 - t_2)}{((s_1 - t_1)^2 + a^2(s_2 - t_2)^2)^{1/2}} Q(s') \psi_{\varepsilon}(s') ds'
		\end{align*}
		Finally, Lemma 2.11 of \cite{NTT} implies that the last two terms are $\varepsilon^3\mathcal{A}_\varepsilon\psi_{\varepsilon}$.
	\end{proof}
	Finally, let us look to the last term of \eqref{int_eq_mani}. By Lemma 5.1 in \cite{NTT}, we know that
	for operator $T_\varepsilon : C_c^\infty(\D) \to {\mathcal D}'(\D)$ defined by the integral kernel 
	$$R(x^\varepsilon (t'; x^*), x^\varepsilon( s'; x^*)) - R(x^*, x^*),$$
	we have $\| T_\varepsilon\|_{ H^{1/2}(\D) ^*\to H^{1/2}(\D)} =O(\varepsilon\log\varepsilon$). Therefore, by using Lemmas \ref{l1}-\ref{l4}, we re-write \eqref{int_eq_mani} in the following way
	\begin{align*}
		-\mathcal{G}(x^*) - C_{\varepsilon,a} =& \frac{\varepsilon}{2\pi}L_a\psi_{\varepsilon} + \frac{1}{4\pi}(H(x^*) - \partial_{\nu}\phi(x^*)) \Phi(x^*)\log\varepsilon\\
		& - \frac{\varepsilon^2}{4\pi}\left((H(x^*) - \partial_{\nu}\phi(x^*)) R_{\log,a} - \frac{\lambda_1(x^*) - \lambda_2(x^*)}{4\pi}R_{\infty,a} + R_{F,a}\right) \psi_{\varepsilon}\\
		& - R(x^*,x^*)\Phi(x^*) + \varepsilon^2 a T_{\varepsilon}\psi_{\varepsilon} + O_{ H^{1/2}(\D)}(\varepsilon\log\varepsilon) + \varepsilon^3\log\varepsilon \mathcal{A}_{\varepsilon}\psi_{\varepsilon}.
	\end{align*}
	This is equivalent to
	\begin{align*}
		\frac{2\pi}{\varepsilon}\left( R(x^*,x^*)\Phi(x^*) -\mathcal{G}(x^*) - C_{\varepsilon,a} - \frac{(H(x^*) - \partial_{\nu}\phi(x^*)) \Phi(x^*)\log\varepsilon}{4\pi}\right) = (L_a - \varepsilon \mathcal{R})\psi_{\varepsilon}\\
		+ \varepsilon a T_{\varepsilon}\psi_{\varepsilon} + O_{ H^{1/2}(\D)}(\log\varepsilon) + \varepsilon^2\log\varepsilon \mathcal{A}_{\varepsilon}\psi_{\varepsilon},
	\end{align*}
	where
	\begin{equation*}
	\mathcal{R}:=  \frac{H(x^*)- \partial_{\nu}\phi(x^*)}{2} R_{\log,a} - \frac{\lambda_1 - \lambda_2}{8} R_{\infty,a} + \frac{1}{2}R_{F,a}.
	\end{equation*}
	Applying $L_a^{-1}$ to both sides, we obtain
	\begin{align}\label{eq}
	\nonumber \frac{2\pi}{\varepsilon}\left( R(x^*,x^*)\Phi(x^*) -\mathcal{G}(x^*) - C_{\varepsilon,a} - \frac{(H(x^*) - \partial_{\nu}\phi(x^*)) \Phi(x^*)\log\varepsilon}{4\pi}\right)L_a^{-1}1\\
	= (I - \varepsilon L_a^{-1} \mathcal{R} + \varepsilon L_a^{-1} T_\varepsilon')\psi_{\varepsilon} + O_{ H^{1/2}(\D)^*}(\log\varepsilon),
	\end{align}
	for some $T'_\varepsilon :  H^{1/2}(\D)^* \to  H^{1/2}(\D)^*$ with operator norm $O(\varepsilon\log\varepsilon)$. As we mentioned in Remark \ref{rem}, $R_{\log,a}$, $R_{\infty,a}$, and $R_{F,a}$ are bounded maps from $H^{1/2}(\D)^*$ to $H^{3/2}(\D)$. Therefore, the right side can be inverted by Neumann series to deduce 
	\begin{align}\label{psi}
	\psi_{\varepsilon} = -\frac{2\pi C_{\varepsilon,a}}{\varepsilon}L^{-1}1 + C_{\varepsilon,a} O_{H^{1/2}(\mathbb{D})^*}(1)+O_{H^{1/2}(\mathbb{D})^*}(\varepsilon^{-1}\log\varepsilon).
	\end{align}
	Let us integrate this over $\mathbb{D}$ and use \eqref{int psi}, then we derive
	\begin{align*}
		-\frac{2\pi C_{\varepsilon,a}}{\varepsilon} \int_{\D} L_a^{-1}1 (s')ds' + C_{\varepsilon,a} O(1) + O(\varepsilon^{-1}\log\varepsilon) = -\frac{\Phi(x^*)}{a\varepsilon^2}.
	\end{align*}
	Note that
	\begin{equation}\label{int L inv}
		\int_{\D} L_a^{-1}1 (s')ds' = \frac{1}{K_a} \int_{\D} \frac{1}{(1 - |s'|^2)^{1/2}}ds' = \frac{2\pi}{K_a},
	\end{equation}
	and hence, the previous equation gives
	\begin{equation}\label{C leading order}
		C_{\varepsilon,a} = \frac{K_a\Phi(x^*)}{4\pi^2 a \varepsilon} + C_{\varepsilon,a}',
	\end{equation}
	with $C_{\varepsilon,a}'=O(\log\varepsilon)$. We put this into \eqref{psi}, to obtain
	\begin{equation}\label{psi_leading_order}
	\psi_{\varepsilon} = -\frac{K_a\Phi(x^*)}{2\pi a\varepsilon^2} L^{-1}_{a}1 + \psi_{\varepsilon}',
	\end{equation}
	where $ \|\psi_\varepsilon'\|_{H^{1/2}(\D;ds')^*} = O(\varepsilon^{-1}\log\varepsilon)$. Let us insert this into \eqref{int psi}, then we obtain
	\begin{equation}\label{int for psi'}
		\int_{\D}\psi_{\varepsilon}'(s')ds' = -\frac{K_a\Phi(x^*)}{2\pi a\varepsilon} \int_{\D}\tilde{Q}(s')L_a^{-1}1(s')ds' + \varepsilon\int_{\D}\tilde{Q}(s')\psi_\varepsilon'(s')ds'
	\end{equation}
	where
	\begin{equation*}
		\tilde{Q}(s') = \frac{1}{\varepsilon}\left((e^{\phi(x^{\varepsilon}(s'))} - e^{\phi(x^*)}) + \varepsilon^2 e^{\phi(x^*)}Q_{\varepsilon}(s')\right)
	\end{equation*}
	and $Q$ is the function involved into volume form \eqref{volume form in epsilon}.
	If we use Taylor expansion to $e^{\phi}$ at $x=x^*$, we obtain
	\begin{align*}
	e^\varphi = e^{\varphi(x^*)}+ \varepsilon g_{x^*}(e^{\varphi(x^*)}\left.\nabla_g \varphi(x)\right\vert_{x=x^*},t_1E_1+t_2E_2)+O(\varepsilon^2).
	\end{align*}
	Therefore
	\begin{align*}
	\int_{\mathbb{D}} e^{\phi(x^{\varepsilon}(t))} [L^{-1}_{a}1](t)dt = &e^{\phi(x^*)}\int_{\mathbb{D}}
	[L^{-1}_{a}](t)dt\\ 
	& + \varepsilon g_{x^*}\left(e^{\varphi(x^*)}\left.\nabla_g \varphi(x)\right\vert_{x=x^*},E_1(x^*)	\right) \int_{\mathbb{D}}t_1 [L^{-1}_{a}1](t)dt\\
	& + \varepsilon g_{x^*}\left(e^{\varphi(x^*)}\left.\nabla_g \varphi(x)\right\vert_{x=x^*},E_2(x^*)	\right) \int_{\mathbb{D}}t_2 [L^{-1}_{a}1](t)dt\\
	& + O(\varepsilon^2).
	\end{align*}
	Noting that 
	\begin{equation*}
	\int_{\mathbb{D}}t_1 [L^{-1}_{a}](t) dt = \int_{\mathbb{D}}t_2 [L^{-1}_{a}1](t)dt = 0,
	\end{equation*}
	we obtain
	\begin{equation}\label{extr_phi_star}
	\int_{\mathbb{D}} e^{\phi(x^{\varepsilon}(t))} [L^{-1}_{a}1](t)dt = e^{\phi(x^*)}\int_{\mathbb{D}}
	[L^{-1}_{a}1](t)dt + O(\varepsilon^2).
	\end{equation}
	Therefore, from \eqref{int for psi'} it follows that 
	\begin{equation}\label{int for psi' last}
		\int_{\D}\psi_{\varepsilon}'(s')ds' = O(1) + O(\log\varepsilon) = O(\log\varepsilon).
	\end{equation}
	Next, we put \eqref{C leading order} and \eqref{psi_leading_order} into \eqref{eq} to obtain
	\begin{align*}
	\frac{2\pi}{\varepsilon}\Bigg( R(x^*,x^*)\Phi(x^*) -\mathcal{G}(x^*) - &C_{\varepsilon,a}' - \frac{(H(x^*) - \partial_{\nu}\phi(x^*)) \Phi(x^*)\log\varepsilon}{4\pi}\Bigg)L_a^{-1}1\\
	&= \frac{K_a\Phi(x^*)}{2\pi a\varepsilon} L^{-1}_{a}\mathcal{R} L_a^{-1}1 - \frac{K_a\Phi(x^*)}{2\pi a\varepsilon} L^{-1}_{a}T_\varepsilon' L_a^{-1}1\\
	& \quad + \psi' - \varepsilon\left(L^{-1}_{a}\mathcal{R} - L^{-1}_{a}T_\varepsilon' \right)\psi' + O_{ H^{1/2}(\D)^*}(\log\varepsilon).
	\end{align*}
	Therefore, recalling that 
	\begin{equation*}
		\|T_\varepsilon'\|_{H^{1/2}(\D)^* \to  H^{1/2}(\D)^*} = O(\varepsilon\log\varepsilon), \quad \|\mathcal{R}\|_{H^{1/2}(\D)^* \to  H^{3/2}(\D)^*} =O(1),
	\end{equation*}
	\begin{equation*}
		\|\psi_\varepsilon'\|_{H^{1/2}(\D;ds')^*} = O(\varepsilon^{-1}\log\varepsilon),
	\end{equation*}
	we derive
	\begin{align*}
		\frac{2\pi}{\varepsilon}\left( R(x^*,x^*)\Phi(x^*) -\mathcal{G}(x^*) - C_{\varepsilon,a} - \frac{(H(x^*) - \partial_{\nu}\phi(x^*)) \Phi(x^*)\log\varepsilon}{4\pi}\right)L_a^{-1}1\\
		= \frac{K_a\Phi(x^*)}{2\pi a\varepsilon} L^{-1}_{a}\mathcal{R} L_a^{-1}1 + \psi_{\varepsilon}' + O_{ H^{1/2}(\D)^*}(\log\varepsilon).
	\end{align*}
	Let us integrate this over $\mathbb{D}$ and take into account \eqref{int for psi' last}, \eqref{int L inv}, then
	\begin{align*}
		C'_{\varepsilon,a} = & - \frac{(H(x^*) - \partial_{\nu}\phi(x^*)) \Phi(x^*)}{4\pi} \log\varepsilon\\
		& + R(x^*, x^*) \Phi(x^*) - \mathcal{G}(x^*) - \frac{K_a^2\Phi(x^*)}{8 \pi^3 a} \int_{\mathbb{D}} L_a^{-1} \mathcal{R}L^{-1}_{a}1(t)dt + O(\varepsilon\log\varepsilon).
	\end{align*}
	Since $L_a^{-1}$ is self-adjoint, we can express the last integral more explicitly:
	\begin{eqnarray*}
		\int_{\D} L_a^{-1} \mathcal{R} L_a^{-1} 1(s') ds' = K_a^{-2} \langle (1-|s'|^2)^{-1/2}, \mathcal{R}(1-|s'|^2)^{-1/2}\rangle.
	\end{eqnarray*}
	Moreover, 
	\begin{equation*}
		\int_{\D} \frac{s_1}{(1 - |s'|^2)^{1/2}} \int_{\D} \frac{1}{((s_1 - t_1)^2 + a^2(s_2 - t_2)^2)^{1/2}}\frac{1}{(1 - |t'|^2)^{1/2}}dt'ds' = 0.
	\end{equation*}
	Indeed, consider the following two changes of variables for the left-hand side
	\begin{eqnarray*}
		(s_1,s_2,t_1,t_2)=(r \cos\alpha,r \sin\alpha, \rho \cos\beta, \rho \sin\beta),\\
		(s_1,s_2,t_1,t_2)=(-r \cos\alpha,r \sin\alpha, -\rho \cos\beta, \rho \sin\beta).
	\end{eqnarray*}
	The results differ by multiplying by $-1$, which means that the left-hand side is 0. Therefore, we know that
	\begin{equation*}
		\int_{\D} L_a^{-1} R_{F,a} L_a^{-1} 1(s') ds' = 0.
	\end{equation*}
	Finally, recalling the definition of $\mathcal{R}$ and the relation between $C_{\varepsilon,a}'$ and $C_{\varepsilon,a}$, we obtain
	\begin{align*}
		C_{\varepsilon,a} = & \frac{K_a\Phi(x^*)}{4\pi^2 a \varepsilon} - \frac{(H(x^*) - \partial_{\nu}\phi(x^*)) \Phi(x^*)}{4\pi} \log\varepsilon\\
		& + R(x^*, x^*) \Phi(x^*) - \mathcal{G}(x^*)\\
		& - \frac{(H(x^*) - \partial_{\nu}\phi(x^*))\Phi(x^*)}{16 \pi^3}\int_{\mathbb{D}} \frac{1}{ (1-|s'|^2)^{1/2}} \int_\D  \frac{\log\left((t_1 - s_1)^2 + a^2 (t_2-s_2)^2 \right)^{1/2}}{ (1-|t'|^2)^{1/2}} dt' ds'\\
		&+\frac{(\lambda_1(x^*) - \lambda_2(x^*))\Phi(x^*)}{64\pi^3}\int_{\mathbb{D}} \frac{1}{ (1-|s'|^2)^{1/2}} \int_\D \frac{(t_1 - s_1)^2 - a^2 (t_2 - s_2)^2}{(t_1 - s_1)^2 + a^2 (t_2 - s_2)^2} \frac{1}{ (1-|t'|^2)^{1/2}} dt' ds'\\
		& + O(\varepsilon\log\varepsilon).
	\end{align*}
	
	In case of the disc, that is $a = 1$, the last two terms explicitly calculated in Lemmas 4.5 and 4.6 in \cite{NTT}. Therefore, we have
	\begin{align}
		C_{\varepsilon} : = C_{\varepsilon,1} = \frac{\Phi(x^*)}{4 a \varepsilon} - \frac{(H(x^*) - \partial_{\nu}\phi(x^*)) \Phi(x^*)}{4\pi} \log\varepsilon + R(x^*, x^*) \Phi(x^*) - \mathcal{G}(x^*)\\
		- \frac{(H(x^*) - \partial_{\nu}\phi(x^*)) \Phi(x^*)}{4\pi} \left( 2\log 2 - \frac{3}{2}\right) + O(\varepsilon\log\varepsilon).
	\end{align}
	Next, let us recall that 
	$$ u_{\varepsilon,a}(x) = \mathbb{E}[\tau_{\Gamma_{\epsilon, a}} | X_0 = x] = \mathcal{G}(x) + C_{\epsilon,a} -\Phi(x^*) G(x^*,x) + r_\epsilon (x)$$
	for each $x\in M\backslash \Gamma_{\epsilon,a}$. Here the remainder $r_\epsilon$ is given by
	\begin{eqnarray}
	\label{r eps remainder}
	r_\epsilon(x) = \int_{\partial M} ( G(x,y) - G(x,x^*))\partial_{\nu} u_\epsilon (y) d_h(y).
	\end{eqnarray}
	Let $K\subset \overline M$ be a compact subset of $\overline M$ which has positive distance from $x^*$ and consider $x\in K$. Writing out this integral in the $x^\epsilon(\cdot; x^*)$ coordinate system and using \eqref{def of psi}, \eqref{volume form in epsilon}, and the expression of $\psi_\epsilon$ derived in \eqref{psi_leading_order}, we get
	\begin{eqnarray}
	r_\epsilon(x) &=& \epsilon  \int_{\D}  \frac{-|M|}{2\pi (1- |s'|^2)^{1/2}} L(x,\epsilon s') (1+\epsilon^2 Q_\epsilon(s')) ds'\\\nonumber
	&+& a\epsilon^3 \int_{\D}\psi'_\epsilon (s') L(x,\epsilon s') (1+ \epsilon^2 Q_\epsilon(s')) ds'
	\end{eqnarray}
	for some function $L(x,s')$ jointly smooth in $(x,s') \in K\times \D$. The second integral formally denotes the duality between $H^{1/2}(\D)^*$ and $H^{1/2}(\D)$. The estimate for $\psi'_\epsilon$ derived in \eqref{psi_leading_order} now gives for any integer $k$ and any compact set $K$ not containing $x^*$, $\| r_\epsilon\|_{C^k(K)}\leq C_{k,K}\epsilon$. This gives us the first parts of Theorems \ref{main theorem disk} and \ref{main theorem}.
	
	Finally, we compute the average expected value over $M$. Let us writte
	\begin{equation*}
		v(x) = \int_{\Gamma_{\epsilon, a}} G (x,y) \partial_{\nu} u_{\varepsilon,a}(y)d_h(y).
	\end{equation*}
	Then
	\begin{equation*}
		v(x) = u_{\varepsilon,a}(x) -\mathcal{G}(x) - C_{\varepsilon,a},
	\end{equation*}
	so that
	\begin{equation*}
	\begin{cases}
		\Delta_{g}v(x) + g(F(x), \nabla_g v(x)) = 0;\\
		v(x)\left.\right|_{\partial M} = \int_{\Gamma_{\epsilon, a}} G_{\partial M} (x,y) \partial_{\nu} u_{\varepsilon,a}(y)d_h(y).
	\end{cases}
	\end{equation*}
	Sicne $G_{\partial M}\in \Psi_{cl}^{-1}(\partial M)$ we know that $v(x)\left.\right|_{\partial M} \in H^{1/2}(\partial M)$. 
	
	Let $\{f_j\}_{j=1}^{\infty}$ be a sequence of smooth functions such that $f_j\to \partial_{\nu} u_{\varepsilon,a}$ in $H^{-1/2}(\partial M)$. Let $\{v_j\}_{j=1}^{\infty}$ be functions which sutisfy 
	\begin{equation*}
	\begin{cases}
	\Delta_{g}v_j(x) + g(F(x), \nabla_g v_j(x)) = 0;\\
	v_j(x)\left.\right|_{\partial M} = \int_{\Gamma_{\epsilon, a}} G_{\partial M} (x,y) f_j(y)d_h(y).
	\end{cases}
	\end{equation*}
	Then $v_j\to v$ in $H^1(M)$. Therefore, we compute 
	\begin{align*}
		\int_M \int_{\Gamma_{\epsilon, a}} G(x,y) \partial_\nu u_\epsilon(y) d_h(y)d_g(x) &= \lim_{j\to \infty}  \int_M \int_{\partial M} G(x,y) f_j(y) d_h(y)d_g(x)\\
		& = \lim_{j\to \infty} \int_{\partial M}  f_j(y) \int_{M} G(x,y) d_g(x) d_h(y)\\
		& = \int_{\Gamma_{\epsilon, a}} \partial_{\nu}u_{\varepsilon,a}(y) \int_{M} G(x,y) d_g(x) d_h(y).
	\end{align*}
	Recalling \eqref{sym for G}, we derive
	\begin{align*}
		\int_{\Gamma_{\epsilon, a}} \partial_{\nu}u_{\varepsilon,a}(y) \int_{M} G(x,y) d_g(x) d_h(y) &= \int_{\Gamma_{\epsilon, a}} \partial_{\nu}u_{\varepsilon,a}(y) \int_{M} e^{\phi(y) - \phi(x)}G(y,x) d_g(x) d_h(y)\\
		& = -\int_{M} e^{\phi(z)} d_g(z) \int_{M} e^{ - \phi(x)}G(x^*,x) d_g(x) + O(\varepsilon)\\
		& = - \Phi(x^*) \int_{M} e^{ \phi(x^*)- \phi(x)}G(x^*,x) d_g(x) + O(\varepsilon)\\
		& = - \Phi(x^*) \int_{M} G(x, x^*) d_g(x) + O(\varepsilon)
	\end{align*}
	Therefore, \eqref{org_int_eq_pr main rres} implies 
	\begin{align*}
	\int_{M} u_{\varepsilon,a}(x) = \int_{M} \mathcal{G}(x) - \Phi(x^*) \int_{M} G(x, x^*) d_g(x) + |M|C_{\varepsilon,a}.
	\end{align*}
	This gives us the second parts of Theorems \ref{main theorem disk} and \ref{main theorem}.

\section{Appendix A -Elliptic Equation for $\mathbb{E}[\tau_{\Gamma} | X_0 = x]$}
	\label{appendix A}
	\renewcommand{\theequation}{A.\arabic{equation}}
	\setcounter{equation}{0}
	In this appendix we show that $u(x) := \mathbb{E}[\tau_{\Gamma} | X_0 = x]$ satisfies the boundary value problem \eqref{main bvp}. This is standard material but we could not find a suitable reference which precisely addresses our setting. As such we are including this appendix for the convenience of the reader.
	
	Let $(M,g,\partial M)$ be an orientable compact connected Riemannian manifold with non-empty smooth boundary oriented by $d_g$. Let us consider the operator
	\begin{equation*}
	u \rightarrow \Delta_gu + g(F,\nabla_gu),
	\end{equation*}
	where $F$ is a force field, which is given by $F=\nabla_g\phi$ for a smooth up to the boundary potential $\phi$. We can re-write this operator in the following way
	\begin{equation*}
	\Delta_{g}^F\cdot := \Delta_g\cdot + g(F,\nabla_g\cdot) = \frac{1}{e^\phi}\mathrm{div}_g(e^\phi \nabla_{g} \cdot).
	\end{equation*}
	Note that $e^\phi$ is a smooth positive function on $M$. Moreover, there exist constants such that
	\begin{equation}\label{e phi bound}
		0<c_0<e^{\phi(x)}<c_1<\infty, \quad x\in M.
	\end{equation}
	According to  \cite{GrigoryanSaloff-Coste}, the operator $\Delta_{g}^F$ is called by weighted Laplace operator and the pair $(M,\mu)$, where $\mu(x) := e^{\phi(x)}d_g(x)$, is called a weighted manifold.
	Note that the operator $\Delta_{g}^{F}$ with initial domain $C_0^\infty(M)$ is essentially self-adjoint in $L^2(M,\mu)$ and non-positive definite. 
	
	Let $(X_t, \mathbb{P}_x)$ be the Brownian motion on $M$ starting at $x$, generated by the weighted Laplace operator $\Delta_{g}^{F}$. Let $\Gamma$ be a geodesic ball on $\partial M$ with radius $\varepsilon > 0$. We denote by $\tau_{\Gamma}$ the first time the Brownian motion $X_t$ hits $\Gamma$, that is
	\begin{equation*}
	\tau_{\Gamma} := \inf \{ t\geq 0: X_t \in \Gamma\}.
	\end{equation*}
	We set
	\begin{equation*}
	\mathcal{P}_{\Gamma}(t,x) := \mathbb{P}[\tau_{\Gamma} \leq t| X_0 = x].
	\end{equation*}
	Let us note that $\mathcal{P}_{\Gamma}(t,x)$ is the probability that the Brownian motion hits $\Gamma$ before or at time $t$, and therefore, satisfies
	\begin{equation}\label{initial_condition}
	\mathcal{P}_{\Gamma}(0,x) = 0, \quad x\in M\setminus \Gamma,
	\end{equation}
	\begin{equation}\label{on_window}
	\mathcal{P}_{\Gamma}(t,x) = 1, \quad (t,x)\in  [0,\infty) \times \Gamma.
	\end{equation}
	
	Note that, for any compact subset $\Gamma\subset M$, it follows
	\begin{equation*}
	\text{Cap}(\Gamma, M) := \inf_{u\in C^{\infty}(\overline{M}), \left.u\right|_{\Gamma} = 1} \int_{M} |\nabla_g u(x)|^2 e^{\phi(x)}d_g(x) =0.
	\end{equation*}
	Note that in \cite{GrigoryanSaloff-Coste} and \cite{ilmavirta}, the authors consider the manifold together with its boundary, and $C_c^{\infty}(M)$, $C_0^{\infty}(M)$ denote the set of smooth (up to the boundary) functions with compact support. In case of compact manifold, these sets coincide with $C^{\infty}(\overline{M})$. This implies that that $(M,\mu)$ is parabolic, that is, the probability that the Brownian motion ever hits any compact set $K$ with non-empty interior is $1$. Since $\Gamma \subset \partial M$ is connected with non-empty interior on $\partial M$, we can extend $M$ to a compact connected Riemannian manifold $\tilde{M}$ such that $\overline{\tilde{M} \setminus M}$ is compact with non-empty interior and $\overline{\tilde{M} \setminus M} \cap M = \Gamma$. Note that, the Brownian motion, starting at any point $M \setminus \Gamma$, hits $\overline{\tilde{M} \setminus M}$ if and only if it hits $\Gamma$. Therefore, the parabolicity condition of $(M, \mu)$ gives 
	\begin{equation}\label{lim_inf}
	\lim_{t \rightarrow \infty} \mathcal{P}_{\Gamma}(t,x) = 1, \quad x \in M.
	\end{equation}
	
	Further, let us define the mean first arrival time $u$, as
	\begin{equation}
	\label{u as expected val}
	u(x) := \mathbb{E}[\tau_{\Gamma} | X_0 = x] := \int_{0}^{\infty} t d \mathcal{P}_{\Gamma}(t,x),
	\end{equation}
	where the integral is a Riemann-Stieltjes integral. To investigate $u$, let us recall some properties of $\mathcal{P}_{\Gamma}$. By Remmark 2.1 in \cite{GrigoryanSaloff-Coste}, it follows that
	\begin{equation*}
	1 - \mathcal{P}_{\Gamma}(t,x) = \left(e^{t\Delta_{mix}^{F}}1\right)(x),
	\end{equation*}
	where $e^{t\Delta_{mix}^{F}}$ is the semigroup with infinitesimal generator $\Delta_{mix}^{F}$, and $\Delta_{mix}^{F}$ is the weighted Laplace operator $\Delta_{g}^F$ corresponding to the Dirichlet boundary condition on $\Gamma$ and Neumann boundary condition on $\partial M \setminus \Gamma$, which is defined as follows
	\begin{align}
	\label{dom of deltamix}
	&\mathrm{D}(\Delta_{mix}^{F}):=\{u\in H^1(M): \; \Delta_g^{F}u \in L^2(M,\mu) \;\left.u\right|_{\Gamma} = 0, \; \left.\partial_{\nu} u\right|_{\Gamma^c} = 0\}\\
	&\Delta_{mix}^F u = \Delta_g^F u \quad u \in \mathrm{D}(\Delta_{mix}^F).
	\end{align}
	Since $e^\phi$ is smooth and satisfies \eqref{e phi bound}, we conclude that, as a set, $L^2(M)=L^2(M, \mu)$. Moreover, for $u\in H^1(M)$, the conditions $\Delta_g^{F}u \in L^2(M,\mu)$ and $\Delta_g u \in L^2(M)$ are equivalent. Therefore, $\mathrm{D}(\Delta_{mix}^{F}) = \mathrm{D}(\Delta_{mix})$, where
	$\Delta_{mix}$ is the classical Laplace operator $\Delta_{g}$ corresponding to the Dirichlet boundary condition on $\Gamma$ and Neumann boundary condition on $\partial M \setminus \Gamma$:
	\begin{align}
	\label{dom of cls deltamix}
	&\mathrm{D}(\Delta_{mix}):=\{u\in H^1(M): \; \Delta_gu \in L^2(M) \;\left.u\right|_{\Gamma} = 0, \; \left.\partial_{\nu} u\right|_{\Gamma^c} = 0\}\\
	&\Delta_{mix} u = \Delta_g u \quad u \in \mathrm{D}(\Delta_{mix}).
	\end{align}
	
	In \eqref{dom of deltamix} and \eqref{dom of cls deltamix}, we define $\partial_\nu u \in H^{-1/2}(\partial M)$ using the same method for defining the Dirichlet to Neumann map. That is, for $u\in H^1(M)$ such that $\Delta_g u \in L^2(M)$, the distribution $\partial_{\nu} u\left. \right|_{\partial M} \in H^{-1/2}(\partial M)$ acts on $f\in H^{1/2}(\partial M)$ via
	\begin{eqnarray}
	\label{def of boundary normal derivative}	
	\langle \partial_{\nu}u\left. \right|_{\partial M} , f\rangle := \int_{M} \Delta_g u(z) \overline{v_f(z)}  d_g(z) + \int_{M} g(\nabla_gu(z),\nabla_g\overline{v_f(z)})  d_g(z), 
	\end{eqnarray}
	where $v_f\in H^1(M)$ is the harmonic extension of $f$.	We say that $\partial_{\nu} u\left. \right|_{\overline{\omega}} = 0$, for non-empty open set $\omega\subset \partial M$, if $\langle \partial_{\nu}u\left. \right|_{\partial M} , f\left. \right|_{\partial M}\rangle = 0$ for all $f \in H^{1/2}(\partial M)$ such that $f|_{\partial M\setminus \overline{\omega}} = 0$.
	Note that if $u$ sufficiently regular, for instance $u\in H^2(M)$, then $\langle\partial_{\nu}u\left. \right|_{\partial M} , f\rangle$ is equal to the boundary integral of $\partial_{\nu}u\left. \right|_{\partial M} $ and $ f$.
	
	Next, we note that 
	\begin{equation*}
		(\Delta_{mix}^{F}u,u)_{L^2(M,\mu)} =  - \int_{M} e^{\phi(z)} g(\nabla_g u(z), \nabla_g u(z)) d_g(z),
	\end{equation*}
	and therefore, \eqref{e phi bound} implies that 
	\begin{equation*}
		c_1(\Delta_{mix}^{F}u,u)_{L^2(M)} < (\Delta_{mix}^{F}u,u)_{L^2(M,\mu)} < c_0(\Delta_{mix}^{F}u,u)_{L^2(M)} 
	\end{equation*}
	for $u\in \mathrm{D}(\Delta_{mix})$, recall that $\mathrm{D}(\Delta_{mix}^{F}) = \mathrm{D}(\Delta_{mix})$. Note that $\Delta_{mix}$ is a self-adjoint operator with discrete spectrum, consisting of negative eigenvalues accumulating at $-\infty$; see for instance Proposition 7.1 in \cite{NTT}. Therefore, the above inequality implies that the spectrum of $\Delta_{mix}^{F}$ consists eigenvalues with finite multiplicity accumulating at $-\infty$. Hence, $\Delta_{mix}^{F}$ satisfies the quadratic estimate
	\begin{equation*}
	\int_0^{\infty} \| t\Delta_{mix}^{F}(1 + t^2(\Delta_{mix}^{F})^2)^{-1} u \|^2_{L^2(M,\mu)}\frac{dt}{t} \leq C\|u\|_{L^2(M,\mu)}^2,
	\end{equation*}
	for some $C>0$ and all $u\in L^2(M,\mu)$; see for instance \cite[p. 221]{McIntosh}. Therefore, $\Delta_{mix}$ admits the functional calculus defined in \cite{Nursultanov}. 
	\begin{remark}
		The functional calculus in \cite{Nursultanov} is defined for a concrete operator, which is denoted by $T$ in the notation used in that article. However, $\Delta_{mix}^{F}$ satisfy all necessary conditions to admit this functional calculus.
	\end{remark}
	Therefore, the semigroup $e^{t\Delta_{mix}^{F}}$, which is contracting by Hille-Yosida theorem \cite[Theorem 8.2.3]{Jost}, can be defined as follows
	\begin{equation*}
	e^{t\Delta_{mix}^{F}} u = \frac{1}{2\pi i}\int_{\gamma_{a,\alpha}} e^{t\zeta}(\zeta - \Delta_{mix}^{F})^{-1} u d\zeta, \qquad u\in L^2(M),
	\end{equation*}
	where $a\in (\tau , 0)$, $\alpha \in (0, \frac{\pi}{2})$, and $\gamma_{a,\alpha}$ is the anti-clockwise oriented curve:
	\begin{equation*}
	\gamma_{a,\alpha}:= \{ \zeta \in \mathbb{C}: \textrm{Re}\zeta \leq a, \text{ } |\textrm{Im}\zeta| = |\textrm{Re}\zeta - a|\tan\alpha\}.
	\end{equation*}
	Let $\varepsilon>0$ such that $a + \varepsilon < 0$. Then $\Delta_{mix}^{F} + \varepsilon$ is also a negative self-adjoint operator, and hence generates contracting semigroup, $e^{t(\Delta_{mix}^{F} + \varepsilon)}$, as above.
	
	By definition, we obtain, for $u \in L^2(M,\mu)$,
	\begin{align}\label{Delta_eps}
	e^{t\Delta_{mix}^{F}} u &= \frac{1}{2\pi i}\int_{\gamma_{a,\alpha}} e^{t\zeta}(\zeta - \Delta_{mix}^{F})^{-1} u d\zeta \\
	&\nonumber = \frac{e^{-t\varepsilon}}{2\pi i} \int_{\gamma_{a,\alpha}} e^{t(\zeta + \varepsilon)}(\zeta + \varepsilon - (\Delta_{mix}^{F} + \varepsilon))^{-1} u d\zeta\\
	&\nonumber = \frac{e^{-t\varepsilon}}{2\pi i} \int_{\gamma_{a + \varepsilon,\alpha}} e^{t\xi}(\xi - (\Delta_{mix}^{F} + \varepsilon))^{-1} u d\xi =  e^{-t\varepsilon} e^{t(\Delta_{mix}^{F} + \varepsilon)}u,
	\end{align}
	where $\gamma_{a + \varepsilon,\alpha} = \gamma_{a,\alpha} + \varepsilon \subset \{\textrm{Re}\xi < 0\}$.
	Let $f_1$ the constant function on $M$ equals $1$. By Theorem 8.2.2 in \cite{Jost}, we know, for $\lambda > 0$,
	\begin{equation*}
	(\lambda - (\Delta_{mix}^{F} + \varepsilon))^{-1}f_1 = \int_{0}^{\infty} e^{-\lambda t} e^{t(\Delta_{mix}^{F} + \varepsilon)}f_1 dt.
	\end{equation*}
	Let us choose $\lambda = \varepsilon$, then, by using \eqref{Delta_eps}, we obtain
	\begin{equation*}
	-(\Delta_{mix}^{F})^{-1} f_1 = \int_{0}^{\infty} e^{-\varepsilon t} e^{t(\Delta_{mix}^{F} + \varepsilon)}f_1 dt = \int_{0}^{\infty} e^{t\Delta_{mix}^{F}} f_1 dt
	\end{equation*}
	and hence, 
	\begin{equation}\label{laplcae_P}
	\int_{0}^{\infty} 1 - \mathcal{P}_{\Gamma}(t,x) dt = -((\Delta_{mix}^{F})^{-1} f_1)(x) < \infty.
	\end{equation}
	Therefore, the dominated convergence theorem implies  
	\begin{equation*}
	\lim_{b \rightarrow \infty} \int_{0}^{b} \left(\mathcal{P}_{\Gamma}(b,x) - \mathcal{P}_{\Gamma}(t,x) \right)dt = \int_{0}^{\infty} 1 - \mathcal{P}_{\Gamma}(t,x) dt <\infty.
	\end{equation*}
	Hence, by using \eqref{u as expected val} and integration by parts, we obtain
	\begin{align*}
	u(x) &= \lim_{b \rightarrow \infty} \left( \mathcal{P}_{\Gamma}(b,x) b -\int_{0}^{b} \mathcal{P}_{\Gamma}(t,x) dt\right) = \lim_{b \rightarrow \infty} \int_{0}^{b} \left(\mathcal{P}_{\Gamma}(b,x) - \mathcal{P}_{\Gamma}(t,x) \right)dt < \infty\\
	& = \int_{0}^{\infty} 1 - \mathcal{P}_{\Gamma}(t,x) dt.
	\end{align*}
	Therefore, by \eqref{laplcae_P}, we obtain
	\begin{equation*}
	\Delta_{mix}^{F} u = -f_1=-1.
	\end{equation*}
	In particular, $u\in \mathrm{D}(\Delta_{mix})$, and hence,
	\begin{equation*}
	u \mid_{\Gamma} = 0, \qquad \partial_\nu u \mid_{\partial M\backslash \Gamma} = 0.
	\end{equation*}
	We see that \eqref{main bvp} is satisfied.

\bibliographystyle{plain}
\bibliography{references}

\setlength{\parskip}{0pt}





\end{document}